\documentclass{aims}
\usepackage{amssymb, amsthm, amsmath}
\usepackage{amsfonts}
\usepackage{mathrsfs}
\usepackage{amsmath}
\usepackage{paralist}
\usepackage{cases}
\usepackage{graphics} 
\usepackage{epsfig} 
\usepackage{graphicx}  \usepackage{epstopdf}
\usepackage[colorlinks=true]{hyperref}
\hypersetup{urlcolor=blue, citecolor=red}
\usepackage[pagewise]{lineno}

\allowdisplaybreaks
\def\currentvolume{X}
 \def\currentissue{X}
  \def\currentyear{200X}
   \def\currentmonth{XX}
    \def\ppages{X--XX}
     \def\DOI{10.3934/xx.xxxxxxx}

\newtheorem{theorem}{Theorem}[section]

\newtheorem{claim}{Claim}[section]
\newtheorem{lemma}[theorem]{Lemma}
\newtheorem{proposition}{Proposition}[section]

\theoremstyle{definition}
\newtheorem{definition}[theorem]{Definition}
\newtheorem{remark}{Remark}[section]

\newcommand{\ep}{\varepsilon}

\newcommand\R{{\mathbb R}}
\newcommand{\fc}{\frac}
\numberwithin{equation}{section}

\title[Schr\"odinger equation] 
      {Local well-posedness and finite time blowup for fourth-order  Schr\"odinger equation with complex coefficient}

\author[X. Liu and T. Zhang]{}

\subjclass{Primary:  35L71; Secondary:  35B30, 35B44.}
\keywords{Fourth-order Schr\"odinger equation; Local well-posedness; Continuous dependence.}

\email{lxmath@zju.edu.cn}
\email{zhangting79@zju.edu.cn}

\thanks{$^*$ Corresponding author: Ting Zhang}

\begin{document}
\maketitle
\centerline{\scshape Xuan Liu and  Ting Zhang$^*$}
\medskip
{\footnotesize
	\centerline{School of Mathematical Sciences, Zhejiang University, Hangzhou 310027, China}
}

\bigskip

\centerline{(Communicated by the associate editor name)}

\begin{abstract}
We consider the fourth-order Schr\"odinger equation
$$
i\partial_tu+\Delta^2 u+\mu\Delta u+\lambda|u|^\alpha u=0,
$$
where $\alpha>0,\mu=\pm1$ or $0$ and  $\lambda\in\mathbb{C}$. Firstly, we prove   local well-posedness in $H^4\left(\R^N\right)$ in both  $H^4$ subcritical and critical case:  $\alpha>0$, $(N-8)\alpha\leq8$.  Then, for any given compact set $K\subset\mathbb{R}^N$, we construct $H^4(\R^N)$ solutions that are defined on $(-T, 0)$ for some $T>0$, and blow up exactly on $K$ at $t=0$.
\end{abstract}

\section{Introduction}
In this paper, we study the following fourth-order nonlinear Schr\"odinger equation of the power type nonlinearity
\begin{numcases}
{\ }
i\partial_tu+\Delta^2 u+\mu\Delta u+\lambda|u|^\alpha u=0,\ t\in\R, x\in \mathbb{R}^N,\label{NLS}\\
u|_{t=0}=\phi,\label{12241}
\end{numcases} where $\alpha>0, \lambda\in\mathbb{C},$ $\mu\in\R$ is essentially given by $\mu=\pm1$ or $\mu=0$, and  $u:\mathbb{R}\times\mathbb{R}^N\rightarrow\mathbb{C}$ is a complex-valued function. Cauchy problem (\ref{NLS})-(\ref{12241}) may be considered as a generation of the  classical Schr\"odinger equation
\begin{equation}\label{11101}
i\partial_tu+\Delta u+\lambda|u|^\alpha u=0
\end{equation}
with the  complex coefficient $\lambda\in\mathbb{C}$,  which in turn is a particular case of the complex Ginzburg-Landau
equation on $\R^N$
\begin{eqnarray}\label{1185}
\partial_{t} u=e^{i \theta} \Delta u+\zeta|u|^{\alpha} u
\end{eqnarray}
where $|\theta| \leq \frac{\pi}{2}$ and $\zeta \in \mathbb{C}$.  Moreover the equation (\ref{1185}) in fact is a generic modulation equation that describes the nonlinear evolution of patterns at near-critical conditions. See for instance \cite{Cr,Mie,Ste}.

The fourth-order Schr\"{o}dinger equation was introduced by Karpman \cite{Karpman1}, and Karpman-Shagalov \cite{Karpman2} to take into
account the role of small fourth-order dispersion terms in the propagation of intense laser beams in a bulk medium with Kerr
nonlinearity. The study of nonlinear fourth-order Schr\"odinger equation with the power type nonlinearity as in (\ref{NLS}) has been attracted a lot of interest in the past decade, see e.g. \cite{Boulenger,Cho,Guo3,Dinh,Miao,Miao2,Pa} and references therein.

One main purpose of this article is to establish the local  well-posedness  for the Cauchy problem (\ref{NLS})-(\ref{12241}) in $H^4\left(\R^N\right)$. In \cite{Pa}, Pausader established the local well-posedness of (\ref{NLS}) in both $H^2\left(\R^N\right)$ subcrtitical and  critical cases, namely, $f(u)=\lambda \left|u\right|^{\alpha}u$ with $\lambda\in\R$ and $0<\alpha, \left(N-4\right)\alpha\le8$. Moreover, he established the global well-posedness and the scattering for the radial datum in the defocusing energy-critical case. In this paper, we show that if the initial datum $\phi\in H^4$, then this regularity can be propagated.  These results are very similar to the classical nonlinear Schr\"odinger case, see e.g.  \cite{Ca4,Kato}.
\begin{theorem}[Subcritical case]\label{T1.00}
	Assume  $\phi\in H^4(\R^N)$, $N\ge1,\lambda\in\mathbb{C},\mu=\pm1$ or $0,\alpha>0$ and $(N-8)\alpha<8$. There exist $T_{\text{max}},T_{\text{min}}\in(0,\infty]$ and a unique  maximum solution $ u \in C ((-T_{\text{min}},T_{\text{max}}),$ $ H^4(\R^N)) $ to  the  Cauchy problem (\ref{NLS})-(\ref{12241}). Moreover, the following properties hold:  \\
	(1) If $T_{\text{max}}<\infty$,  then
	\begin{equation}\label{bl1}
	\lim_{t\uparrow T_{\text{max}}}\|u(t)\|_{H^4}=\infty.\end{equation}
	A corresponding conclusion is reached if $T_{\text{min}}<\infty$. \\
	(2)  The solution $u\in H_{loc}^{1,q}((-T_{\text{min}},T_{\text{max}}),L^r(\R^N)) $ for any biharmonic admissible pair $(q,r)\in\Lambda_b$. (See Section \ref{section 02} for the definition of $\Lambda_b$.)\\
	(3) The solution $u$ depends continuously on $\phi$ in the following sense:  If $\phi_n\rightarrow\phi$ in $H^4(\R^N)$ and  $u_n$ denotes the solution of (\ref{NLS}) with the initial value $\phi_n$, then $u_n\rightarrow u$ in $C([-B,A],H^4(\R^N))$ and in $H^{1,q}((-B,A),L^r(\R^N))$  for any biharmonic admissible pair $(q,r)\in\Lambda_b$    and every $-T_{\text{min}}<-B<0<A<T_{\text{max}}.$
\end{theorem}
\begin{theorem}[Critical case]\label{T1.01}
	Assume $\phi\in H^4(\R^N)$, $N\ge9$, $\lambda\in\mathbb{C},\mu=\pm1$ or $0$ and $\alpha=\frac{8}{N-8}$.  There exist $T_{\text{max}},T_{\text{min}}\in(0,\infty]$ and a unique  maximum solution $ u \in C ((-T_{\text{min}},T_{\text{max}}),$ $ H^4(\R^N)) $ to  the  Cauchy problem (\ref{NLS})-(\ref{12241}). Moreover, the following properties hold:  \\
	(1) If 	$T_{\text{max}}<\infty$,  then \begin{equation}\label{bl3}
	\|u(t)\|_{L^{\fc{2N-8}{N-8}}([0,T_{\text{max}}),L^{\fc{2N(N-4)}{(N-8)^2}})}=\infty.
	\end{equation}
	A corresponding conclusion is reached if $T_{\text{min}}<\infty$. \\
	(2)  The solution $u\in H_{loc}^{1,q}((-T_{\text{min}},T_{\text{max}}),L^r(\R^N)) $ for any biharmonic admissible pair $(q,r)\in\Lambda_b$. \\
	(3) The solution $u$ depends continuously on $\phi$ in the following sense:  If $\phi_n\rightarrow\phi$ in $H^4(\R^N)$ and  $u_n$ denotes the solution of (\ref{NLS}) with the initial value $\phi_n$, then $u_n\rightarrow u$ in $C([-B,A],H^4(\R^N))$ and in $W^{1,q}((-B,A),L^r(\R^N))$  for any biharmonic admissible pair $(q,r)\in\Lambda_b$    and every $-T_{\text{min}}<-B<0<A<T_{\text{max}}.$
	
\end{theorem}
We prove Theorems \ref{T1.00} and \ref{T1.01} in the spirit of  \cite{Kato,Kato2}. See also Cazenave, Fang and Han \cite{Ca4}. The fundamental ingredient  is the choice of the metric space where the Banach fixed-point argument can be applied. We note  that
obtaining $H^4$  estimates  by  differentiating   the  equation four times in  space  would  require  that  the nonlinearity is sufficiently smooth.  Instead,
we  differentiate the equation  once  in time  and  then  deduce  $H^4$ estimates  by using  the
equation and the estimates of $\partial_tu$.  See Sections \ref{s4} and \ref{s5} for more details.

Another main result of this paper concerns the finite time blowup of the fourth-order  Schr\"odinger equation (\ref{NLS}).
The existence of blowup solutions for the equation (\ref{NLS}) has been strongly supported by a series of numerical studies done by Fibich, Ilan and Papanicolaou \cite{Fibich} for mass-critical and mass-supercritical powers $\alpha\geq\frac8N$. Recently, it has been proved rigorous in Boulenger \cite{Boulenger}, under the natural criteria from the well-known blowup results for the classical nonlinear Schr\"odinger equation, that the blowup solutions exist for the radial data in $H^2(\R^N)$ with the negative initial energy.
For more blowup results for the fourth-order Schr\"odinger equation, we refer to \cite{BF1,BF2,BF3,Cho,Dinh}  and references therein.

In this paper, we establish the following  finite time blowup result for the fourth-order Schr\"odinger equation (\ref{NLS}).
\begin{theorem}\label{T1.1}
	Under the conditions
	\begin{equation}\label{71512}
	\alpha>0,\ (N-8)\alpha\leq8,\ \text{Im }\lambda<0,
	\end{equation}
	for any nonempty compact subset  $K\subset\mathbb{R}^N$,   there exist  $S\in(-1, 0)$ and a solution $u\in C ([S,  0),$ $  H^4(\mathbb{R}^N))\cap C^1([S, 0), L^{2}(\mathbb{R}^N))$ to the equation (\ref{NLS})
	which blows up at time 0 exactly on $K$ in the following sense. \\
	$(1)$ If $x_0\in K$ then for any $r>0$,
	\begin{equation}\label{1.4}
	\lim_{t\uparrow0}\|u(t)\|_{L^2(|x-x_0|<r)}=\infty.
	\end{equation}\\
	$(2)$ If $U$ is an open subset of $\mathbb{R}^N$ such that $K\subset U$,  then
	\begin{equation}\label{1.5}
	\lim_{t\uparrow0}\|\Delta^2 u(t)\|_{L^2(U)}=\infty, \qquad\lim_{t\uparrow0}\|\partial_tu(t)\|_{L^2(U)}=\infty.\\
	\end{equation}
	$(3)$ If  $\Omega$ is an open subset of $\mathbb{R}^N$ such that $\overline\Omega\cap K={\O}$,  then
	\begin{equation}\label{1.6}
	\sup_{t\in(S,  0]}\|u(t)\|_{H^4(\Omega)}<\infty, \qquad\sup_{t\in(S,  0]}\|\partial_tu(t)\|_{L^2(\Omega)}<\infty.
	\end{equation}
\end{theorem}
\begin{remark}
	It follows from (\ref{1.4}) and (\ref{1.5}) that both $\|u(t)\|_2,  \|\Delta^2 u(t)\|_2$ and $\|\partial_t u(t)\|_2$ blow up when $t\uparrow0$. Moreover, the estimate (\ref{1.4}) can be refined. More precisely,  it follows from (\ref{4.11}), (\ref{4.6e}) and (\ref{4.8}) that
	\begin{equation*}
	(-t)^{-\frac1\alpha+\frac{N}{2k}}\lesssim\|u(t)\|_{L^2(|x-x_0|<r)}\lesssim(-t)^{-\frac1\alpha}
	\end{equation*}
	where $k> {N\alpha}$ is given by 
	({\ref{k}}).
\end{remark}
Similar blowup results has been established for the classical Schr\"odinger equation. With the restriction $\alpha\ge2$,  it is proved in Cazenave, Martel and Zhao \cite{Ca8} that under the assumption  $(N-2)\alpha\le4$ and $\text{Im}\lambda=-1$, the  finite time blowup occurs.  More precisely, they introduce the ansatz $U_0$ satisfying
$$
i\partial_tU_0+\lambda|U_0|^\alpha U_0=0,
$$
 then using energy estimates and compactness arguments to complete the proof. After that,   by refining the initial ansatz $U_0$ inductively,  the blow-up result is extended to the whole range of $H^1$ subcritical powers and arbitrary  $\text{Im}\lambda<0$ in Cazenave, Han and Martel \cite{Ca2}. This  result is then extended into $H^2$-subcritical case: $\alpha>0$ and $(N-4)\alpha<4$ in \cite{xuan}  with  the additional technical assumption $-\text{Im}\lambda>\frac\alpha2|\text{Re}\lambda|$.

We prove Theorem \ref{T1.1} by applying  the strategy of \cite{Ca8}. More precisely,  we consider the sequence $\{u_n\}_{n\geq1}$ of solutions of (\ref{NLS}) with the initial datum $u_n(-\frac1n)=U_J(-\frac1n)$,  where $U_J$ is a refined blowup profile defined in Lemma \ref{L2.2J}. Since $U_J(-\frac1n)\in H^4(\R^N)$ by (\ref{k}) and (\ref{2.31J}), it follows from Theorems \ref{T1.00} and \ref{T1.01} that $u_n$ is defined on $(s_n, -\frac1n)$ for some $s_n<-\frac1n$. We then define  $\varepsilon_n(t)=u_n(t)-U_J(t)$. Next,  by the energy arguments, we show that $\{\varepsilon_n\}_{n\geq1}$ is uniformly bounded in $L^\infty((S, \tau),  H^4)$ ($S$ is given by Proposition \ref{P3.1}) for any $\tau\in(S, 0)$. In Section \ref{section 4},  we find $\varepsilon\in$ $L^\infty((S, 0),  H^2)\cap W^{1, \infty}((S, 0), L^{2})$ and  a subsequence of $\{\varepsilon_n\}_{n\geq1}$ that  converges weakly to $\varepsilon$   by the compactness argument. Next, we  define  $u(t)=U_J(t)+\varepsilon(t)$, which turns out to be a $H^4$ solution to the equation  (\ref{NLS}). Finally,  note that $\varepsilon$ is bounded in$L^\infty([S, 0),$ $ H^4(\mathbb{R}^N))\cap C^1([S, 0),$ $ L^2(\mathbb{R}^N))$ and $U_J$ blows up at time $0$ exactly on $K$,  we deduce that $u(t)$ also blows up at time $0$ exactly on $K$.

The solution $u$ given by Theorem \ref{T1.1} blows up at $t=0$ like the function $U_J$ defined in Lemma \ref{L2.2J}. Since the function $U_0$ defined by (\ref{2.03}) satisfying $i\partial_tU_0+\lambda|U_0|^\alpha U_0 =0$,   and $U_J$ is a refinement of $U_0$,  we see that the solution $u$ displays an ODE-type blowup. We recall that there are many ODE-type blowup results for several other nonlinear equations,  refer to \cite{Col, Ca7,Me2, No} for  results in the parabolic context,  refer to \cite{Al, Me3, Sp} for the nonlinear wave  equations.

The rest of the paper is organized as follows.  In Section \ref{section 02}, we fix notations and recall preliminary results. In section \ref{s3}, we establish the nonlinear estimate. In Sections \ref{s4} and \ref{s5}, we prove Theorem \ref{T1.00} in the case $N\ge9$ and Theorem \ref{T1.01} respectively. In Section \ref{section 2}, we introduce the blow-up ansatz and the corresponding estimates. Section \ref{section 3} is devoted to the construction of a sequence of solutions of (\ref{NLS}) close to the blow-up ansatz and some \textit{a priori} estimates of the approximate solutions. In Section \ref{section 4},  we complete the proof of Theorem \ref{T1.1} by passing to the limit in the approximate solutions. Finally, an appendix is devoted to the proof of unconditional uniqueness theorem and Theorem \ref{T1.00} in the case $1\leq N\leq8.$
\section{Preliminary}\label{section 02}
If $X, Y$ are nonnegative quantities, we sometimes use $X\lesssim Y$ to denote the estimate $X\leq CY$ for some positive constant $C$. Pairs of conjugate indices are written as $p$ and $p', 1\leq p\leq\infty$, $\frac1p+\frac1{p'}=1$. We use $L^p (\mathbb{R}^N)$ to denote the Lebesgue space of functions $f :\mathbb{R}^N \rightarrow\mathbb{C}$ whose norm
\begin{gather}
\|f\|_ {L_x^p}:=
\left(\int_{\mathbb{R}^N}|f(x)|^p dx\right)^{\frac1p}
\end{gather}
where $1\leq p<\infty$ is finite, with the usual modifications when $p =\infty$. Given $k\in \mathbb N,1\le p\le \infty $, we use $H^{k,p}\left(\R^N\right)$ or $H^{k,p}$ to denote the usual Sobolev space. We also use the space-time Lebesgue spaces $L^\gamma \left(I,L^\rho\right)$
which are equipped with the
norm
\begin{gather}\notag
\|f\|_{L^\gamma(I;L^\rho(\mathbb{R}^N))}:=\left(\int_I\|f\|_{L_x^\rho}^\gamma dt\right)^{\frac1\gamma}
\end{gather}
for any space-time slab $I\times\mathbb{R}^N$, with the usual modification when either $\gamma$ or $\rho$ is infinity. For $1\le q<\infty$, Banach space $X$, we also introduce the vector-valued Lebesgue spaces $L^q \left(I,X\right)$ which are equipped with the norm
\begin{gather}\notag
\|u\|_{L^q(I,X)}:=\left(\int_I\|u\|_{X}^q dt\right)^{\frac1q}
\end{gather}
for any time slab $I\subset\mathbb{R}$, with the usual modification when $q$ is infinity. Furthermore, the Banach space of functions $u \in H^{1, p}(I, X)$ satisfies $u, \partial_{t} u \in L^{p}(I, X)$. For $\alpha>0$, we denote
\begin{eqnarray*}
	1_{\alpha>1}=\begin{cases}
		0,   & \text{if } 0<\alpha\leq 1, \\
		1,   & \text{if } \alpha>1,
	\end{cases}
	\qquad\text{ and }\qquad 	1_{0<\alpha<1}=\begin{cases}
		1,   & \text{if } 0<\alpha< 1, \\
		0,   & \text{if } \alpha\ge1.
	\end{cases}
\end{eqnarray*}

Following standard notations, we introduce Schr\"odinger admissible pairs as well as the corresponding Strichartz's estimate for the fourth-order Schr\"odinger equation.

\begin{definition}\label{bpair}
	A pair of Lebesgue space exponents $(\gamma, \rho)$ is called  biharmonic Schr\"odinger admissible for the equation (\ref{NLS}) if  $(\gamma, \rho)\in \Lambda_b$ where
	\begin{equation*}
	\Lambda_b=\{(\gamma, \rho):2\leq \gamma, \rho\leq\infty, \   \frac4\gamma+\frac N\rho=\frac N2, \  (\gamma, \rho, N)\neq(2, \infty, 4)\},
	\end{equation*}
\end{definition}
\begin{lemma}[Strichartz estimates for fourth-order NLS, \cite{Pa}]\label{L2.2S}
	Suppose that $(\gamma,\rho), $ $ (a,b)\in\Lambda_b $ are any two biharmonic admissible pairs, then for any function $u\in L^2(\mathbb{R}^N)$ and $h\in L^{a'}(I,L^{b'}(\mathbb{R}^N))$(if $\mu>0$, suppose also $|I|\leq1$), we have
	\begin{gather}\label{sz}
	\|e^{it(\Delta^2+\mu\Delta)}u\|_{L^\gamma(I; L^\rho)}\leq C\|u\|_{L^2},
	\end{gather}
	\begin{equation}\label{SZ}
	\|\int_0^te^{i(t-s)(\Delta^2+\mu\Delta)}h\ ds \|_{L^\gamma(I; L^\rho)}\leq C\|h\|_{L^{a'}(I;L^{b'})}
	\end{equation}
	where the integrates with respect to time are all over compact interval $I\subset\mathbb{R}.$
\end{lemma}

In the rest of  this paper, we fix $\lambda\in\mathbb{C}$, $\mu=\pm1$ or $0$. We define
\begin{equation}\label{7181}
\gamma=\frac{8(\alpha+2)}{(N-8)\alpha},\ \rho=\frac{N(\alpha+2)}{N+4\alpha},
\end{equation}
when $N\ge9$ and $0<\alpha\le \frac{8}{N-8}$. Then it is easy to check that  $(\gamma,\rho)\in\Lambda_b$ is a biharmonic admissible pair, which will be frequently used along the paper.
Moreover we define the map
\begin{eqnarray}\label{12113}
(Su)(t)=e^{it(\Delta^2+\mu\Delta)}\phi+i\lambda \int_{0}^{t}e^{i\left(t-s\right)(\Delta^2+\mu\Delta)}[\left|u\right|^{\alpha}u]\left(s\right)\mathrm{d}s.
\end{eqnarray}
Note that $Su$ satisfies
\begin{eqnarray}\label{1217}
\begin{cases}
i\partial_t\left(Su\right)+\Delta^2\left(Su\right)+\mu\Delta \left(Su\right)+\lambda \left|u\right|^{\alpha}u=0,\\
\left(Su\right)(0)=\phi
\end{cases}
\end{eqnarray}
and that
\begin{eqnarray}\label{NLS2}
\partial_t \left(Su\right)&=&i e^{it(\Delta^2+\mu\Delta)}[\left(\Delta^2+\mu\Delta\right)\phi+\lambda \left|\phi\right|^{\alpha}\phi]\notag\\
&&+i\lambda \int_{0}^{t}e^{i\left(t-s\right)(\Delta^2+\mu\Delta)}\partial_s[\left|u\right|^{\alpha}u]\left(s\right)\mathrm{d}s.
\end{eqnarray}
Given  any $\phi\in H^4$ and $0<T<\infty $, we define
\begin{eqnarray}\label{1213}
F(\phi,T)&=&\|e^{it(\Delta^2+\mu\Delta)}(\Delta^2+\mu\Delta)\phi\|_{L^{\gamma}\left([0,T],L^\rho\right)}+\|e^{it(\Delta^2+\mu\Delta)}\left|\phi\right|^{\alpha}\phi\|_{L^{\gamma}\left([0,T],L^\rho\right)}\notag\\
&&+\left\|e^{it(\Delta^2+\mu\Delta)}\phi\right\|_{L^\gamma \left([0,T], H^{4,\rho}\right)}.
\end{eqnarray}
It follows from Strichartz's estimate (\ref{sz}) and Sobolev's embedding $H^4 \hookrightarrow L^{2\left(\alpha+1\right)}$ that
\begin{eqnarray}\label{12111}
F\left(\phi,T\right)&\lesssim& \left\|\left(\Delta^2+\mu\Delta\right)\phi\right\|_{L^2}+\left\|\left|\phi\right|^{\alpha}\phi\right\|_{L^2}+\left\|\phi\right\|_{H^4}\notag\\
&\lesssim &\left\|\phi\right\|_{H^4}+\left\|\phi\right\|_{H^4}^{\alpha+1}.
\end{eqnarray}
So that we have
\begin{eqnarray}
\lim_{T\to 0}F\left(\phi,T\right)=0
\end{eqnarray}
by the dominate convergence theorem. Moreover, it follows from Strichartz's estimate (\ref{sz}) and Sobolev's embedding $H^4 \hookrightarrow L^{2\left(\alpha+1\right)}$ again that the map $\left(\phi,T\right)\mapsto F\left(\phi,T\right)$ is continuous $H^4\times (0,\infty ]\to(0,\infty )$. Therefore, if $\phi_n\longrightarrow\phi$ in $H^4$ as $n\longrightarrow \infty $, then
\begin{eqnarray}
\sup_{n\ge1}F\left(\phi_n,T\right)\underset{T\downarrow0} \longrightarrow0.
\end{eqnarray}
\section{Nonlinear estimates}\label{s3}
The goal of this section is to establish the  nonlinear estimates that we will need  to  prove Theorems \ref{T1.00} and \ref{T1.01}.
Throughout this section,  we fix  $\lambda\in\mathbb{C},\mu=\pm1$ or $0, N\ge9,0<\alpha\le\fc8{N-8}$
and  $I=[0, T]$ with $T>0$.

Before starting the Lemmas, it is useful  to introduce  following numbers.
 Let $q_0,p_0$ be given by the equations
\begin{eqnarray}\label{12171}
\frac{\alpha+2}{2\left(\alpha+1\right)}=\frac{\alpha+1}{q_0}+\frac{1}{\rho},
\end{eqnarray}
\begin{eqnarray}\label{12201}
\frac{1}{\rho}=\frac{\alpha}{2(\alpha+1)}+\frac{1}{p_0}.
\end{eqnarray}
 Since $\rho=\frac{N\left(\alpha+2\right)}{N+4\alpha}$ by (\ref{7181}),  it is easily seen that
\begin{eqnarray}\label{12172}
\fc{\alpha+2}{2(\alpha+1)}-(\alpha+1)\left(\fc1\rho-\fc4N\right)-\fc1\rho
=\fc{8-(N-8)\alpha}{2(\alpha+1)}\ge0,
\end{eqnarray}
\begin{eqnarray}\label{12173}
\frac{\alpha+2}{2\left(\alpha+1\right)}-\left(\alpha+1\right)\frac{1}{\rho}-\frac{1}{\rho}=-\frac{4\alpha}{N}<0,
\end{eqnarray}
and
\begin{eqnarray}\label{12202}
\frac{1}{p_0}-\frac{1}{\rho}+\frac{4}{N}=\frac{8-\left(N-8\right)\alpha}{2N\left(\alpha+1\right)}\ge0.
\end{eqnarray}
 It follows from  (\ref{12171})--(\ref{12202}) that  $\frac{1}{\rho}\ge\frac{1}{q_0}\ge\frac{1}{\rho}-\frac{4}{N}$,
and $\frac{1}{\rho}>\frac{1}{p_0}\ge\frac{1}{\rho}-\frac{4}{N}$.
Therefore, we deduce the embedding
\begin{eqnarray}\label{a1}
\left\|u\right\|_{L^{q_0}\cap L^{p_0}}\lesssim \left\|u\right\|_{H^{4,\rho}}.
\end{eqnarray}
\begin{lemma}\label{l1}
	Assume that  $u^1,u^2\in  L^\gamma\left(I,H^{4,\rho}\right)\cap H^{1,\gamma}\left(I,L^\rho\right)$ are two solutions to the equation (\ref{NLS}) with initial   data $\phi^1,\phi^2$ respectively,  there exist $\rho<\widetilde{\rho}_1,\widetilde{\rho}_2<\infty$ such that  for any $R>0$ we have
	\begin{eqnarray}\label{12301}
		&&\sup_{(q,r)\in\Lambda_b}\|Su^1-Su^2\|_{H^{1,q}\left(I,L^r\right)}\notag\\
		&\lesssim&\left(\|\phi^1\|_{H^4}^\alpha+\|\phi^2\|_{H^4}^\alpha+1\right)\|\phi^1-\phi^2\|_{H^4}\notag\\
		&&+T^{1-\fc{\alpha+2}\gamma} \left(\|u^1\|_{L^\gamma \left(I,H^{4,\rho}\right)}^\alpha+\|u^2\|_{L^\gamma \left(I,H^{4,\rho}\right)}^\alpha\right)\|u^1-u^2\|_{H^{1,\gamma} \left(I,L^\rho\right)}\notag\\
		&&+T^{1-\frac{\alpha+2}{\gamma }}\left\|\left(\partial_tu^2\right)^R\right\|_{L^\gamma \left(I,L^\rho\right)}\left(\left\|u^1\right\|_{L^\gamma \left(I,H^{4,\rho}\right)}^\alpha+\left\|u^2\right\|_{L^\gamma \left(I,H^{4,\rho}\right)}^\alpha\right)\notag\\
		&&+1_{0<\alpha\le1} T^{1-\fc{\alpha+2}\gamma+\fc{\widetilde\rho_1-\rho}{\gamma\widetilde\rho_1}}R^{\fc{\widetilde\rho_1-\rho}{\widetilde\rho_1}}\|\partial_tu^2\|_{L^{\gamma}\left(I,L^\rho\right)}^{\fc{\rho}{\widetilde\rho_1}}\|u^1-u^2\|_{L^\gamma\left(I,L^\rho\right)}^\alpha\notag\\
		&&+ 1_{\alpha>1}T^{1-\fc{\alpha+2}\gamma+\fc{\widetilde\rho_2-\rho}{\gamma\widetilde\rho_2}}R^{\fc{\widetilde\rho_2-\rho}{\widetilde\rho_2}}\|\partial_tu^2\|_{L^{\gamma}\left(I,L^\rho\right)}^{\fc{\rho}{\widetilde\rho_2}}\left(\|u^1\|_{L^\gamma \left(I,H^{4,\rho}\right)}^{\alpha-1}+\|u^2\|_{L^\gamma \left(I,H^{4,\rho}\right)}^{\alpha-1}\right)\notag\\
		&&\qquad\times\|u^1-u^2\|_{L^\gamma \left(I,L^\rho\right)}.
	\end{eqnarray}
	where
	\begin{eqnarray*}
		\left(\partial_tu^2\right)^R=\begin{cases}
			\left|\partial_tu^2\right|^{},\qquad\text{if }\left|\partial_tu^2\right|\ge R,\\
			0,\qquad\qquad\text{if }\left|\partial_tu^2\right|< R.\
		\end{cases}
	\end{eqnarray*} Moreover, we have
	\begin{eqnarray}\label{1214}
	\left\|Su^1\right\|_{H^{1,\gamma }\left(I,L^\rho\right)}\lesssim F\left(\phi^1,T\right)+T^{1-\frac{\alpha+2}{\gamma }}\left\|u^1\right\|_{L^\gamma \left(I,H^{4,\rho}\right)}^\alpha \left\|u^1\right\|_{H^{1,\gamma }\left(I,L^\rho\right)}.
	\end{eqnarray}
\end{lemma}
\begin{proof}
	Firstly, we prove (\ref{1214}).  It follows from  equations (\ref{12113}), (\ref{NLS2}), Strichartz estimate (\ref{SZ}), H\"older's inequality  and Sobolev's embedding $H^{4,\rho}\hookrightarrow L^{\frac{\rho\alpha}{\rho-2}}$ that
	\begin{eqnarray}\label{1215}
	\left\|Su^1\right\|_{H^{1,\gamma }\left(I,L^\rho\right)}&\lesssim& F\left(\phi^1,T\right)+\left\|\left|u^1\right|^{\alpha}u^1\right\|_{H^{1,\gamma '}\left(I,L^{\rho'}\right)}\notag\\
	&\lesssim &F\left(\phi^1,T\right)+T^{1-\frac{\alpha+2}{\gamma }}\left\|u^1\right\|_{L^\gamma \left(I,H^{4,\rho}\right)}^\alpha \left\|u^1\right\|_{H^{1,\gamma }\left(I,L^\rho\right)},
	\end{eqnarray}
	which gives (\ref{1214}).

	We now prove (\ref{12301}). Using the same method as that used to derive  (\ref{1215}), we obtain
\begin{eqnarray}\label{1216}
&&\sup_{\left(q,r\right)\in\Lambda_b}\left\|Su^1-Su^2\right\|_{L^q\left(I,L^r\right)}\notag\\
&\lesssim& \left\|\phi^1-\phi^2\right\|_{L^2}+\left\|\left|u^1\right|^{\alpha}u^1-\left|u^2\right|^{\alpha}u^2\right\|_{L^{\gamma '}\left(I,L^{\rho'}\right)}\\
&\lesssim &\left\|\phi^1-\phi^2\right\|_{L^2}+T^{1-\frac{\alpha+2}{\gamma }}\left(\left\|u^1\right\|_{L^\gamma \left(I,H^{4,\rho}\right)}^\alpha+\left\|u^2\right\|_{L^\gamma \left(I,H^{4,\rho}\right)}^{\alpha}\right) 
  \left\|u^1-u^2\right\|_{L^\gamma \left(I,L^\rho\right)},\notag
\end{eqnarray}
	and
	\begin{eqnarray}\label{11171}
	&&\sup_{(q,r)\in\Lambda_b}\|\partial_t\left(Su^1\right)-\partial_t\left(Su^2\right)\|_{L^q\left(I,L^r\right)} \\
	&\lesssim&\left\||\phi^1|^\alpha \phi^1-|\phi^2|^\alpha\phi^2\right\|_{L^2}+\|\phi^1-\phi^2\|_{H^4}\notag 
	 +\left\|\partial_t \left(\left|u^1\right|^{\alpha}u^1-\left|u^2\right|^{\alpha}u^2\right)\right\|_{L^{\gamma'}\left(I,L^{\rho'}\right)}
	\end{eqnarray}
	where we used the Strichartz's estimate (\ref{sz}). Applying H\"older's inequality and Sobolev's embedding  $H^4\hookrightarrow L^{2(\alpha+1)}$, we conclude that
	\begin{eqnarray}\label{11303}
	\left\||\phi^1|^\alpha \phi^1-|\phi^2|^\alpha\phi^2\right\|_{L^2}&\lesssim &\left(\|\phi^1\|_{L^{2(\alpha+1)}}^\alpha+\|\phi^1\|_{L^{2(\alpha+1)}}^\alpha\right)\|\phi^1-\phi^2\|_{L^{2(\alpha+1)}}\notag\\
	&\lesssim& \left(\|\phi^1\|_{H^4}^\alpha+\|\phi^2\|_{H^4}^\alpha\right)\|\phi^1-\phi^2\|_{H^4}.
	\end{eqnarray}
	Our next step is to control the last term in (\ref{11171}).
	Note that  $\partial_t \left(|u|^\alpha u\right)=\frac{\alpha+2}2|u|^\alpha\partial_t u+\frac\alpha2|u|^{\alpha-2}u^2\partial_t\overline u$, we have by applying the inequality (\ref{121710}) below
	\begin{eqnarray*}
		\left|\partial_t \left(\left|u^1\right|^{\alpha}u^1-\left|u^2\right|^{\alpha}u^2\right)\right|\lesssim \left|u^1\right|^{\alpha}\left|\partial_tu^1-\partial_tu^2\right|+F(u^1,u^2)\left|\partial_tu^2\right|,
	\end{eqnarray*}
	where
	\begin{eqnarray}\label{11302}
	F(u^1,u^2)=\begin{cases}
	\left(\left|u^1\right|^{\alpha-1}+\left|u^2\right|^{\alpha-1}\right)\left|u^1-u^2\right| \text{ if } \alpha\ge1,\\
	\left|u^1-u^2\right|^{\alpha} \text{ if } 0<\alpha\le1.
	\end{cases}
	\end{eqnarray}
	Using H\"older's inequality and the  embedding  $H^{4,\rho}\hookrightarrow L^{\fc{\rho\alpha}{\rho-2}}$, we deduce that
	\begin{eqnarray}\label{11301}
	&&\left\||u^1|^\alpha\left(\partial_tu^1-\partial_tu^2\right)\right\|_{L^{\gamma'}\left(I,L^{\rho'}\right)}\notag\\
	&\lesssim& T^{1-\fc{\alpha+2}\gamma}\|u^1\|_{L^\gamma \left(I,H^{4,\rho}\right)}^\alpha \left\|\partial_tu^1-\partial_tu^2\right\|_{L^{\gamma}\left(I,L^{\rho}\right)}.
	\end{eqnarray}
	By (\ref{1216})--(\ref{11301}), the proof of  (\ref{12301})  reduces to control  $\left\|\partial_tu^2F(u^1,u^2)\right\|_{L^{\gamma'}\left(I,L^{\rho'}\right)}$.

	We decompose $\left|\partial_tu^2\right|=\left(\partial_tu^2\right)^R+\left(\partial_tu^2\right)_R$, where $\left(\partial_tu^2\right)_R=\left|\partial_tu^2\right|-\left(\partial_tu^2\right)^R$. Since $F(u^1,u^2)\lesssim \left|u^1\right|^{\alpha}+\left|u^2\right|^{\alpha}$ by (\ref{11302}), it  follows  from H\"older's inequality and Sobolev's embedding $H^{4,\rho}\hookrightarrow L^\frac{\rho\alpha}{\rho-2}$ that
	\begin{eqnarray*}
		&&\left\|\left(\partial_tu^2\right)^RF\left(u^1,u^2\right)\right\|_{L^{\gamma' }\left(I,L^{\rho'}\right)}\notag\\
		&\lesssim& T^{1-\frac{\alpha+2}{\gamma }}\left\|\left(\partial_tu^2\right)^R\right\|_{L^\gamma \left(I,L^\rho\right)}\left(\left\|u^1\right\|_{L^\gamma \left(I,H^{4,\rho}\right)}^\alpha+\left\|u^2\right\|_{L^\gamma  \left(I,H^{4,\rho}\right)}^\alpha\right).
	\end{eqnarray*}
	Our final step is  to estimate $\left\|\left(\partial_tu^2\right)_RF\left(u^1,u^2\right)\right\|_{L^{\gamma'}\left(I,L^{\rho'}\right)}$. We consider  two cases.
	
	Suppose   $0<\alpha\le1$. Let $\widetilde\rho_1=\frac{\rho}{\rho-\left(\alpha+1\right)}$.   Then it is easy to check that $\fc1{\rho'}=\fc\alpha\rho+\fc1{\widetilde\rho_1}$ and  $\rho<\widetilde \rho_1<\infty$.  This together with $\left|\left(\partial_tu^2\right)_R\right|\le R$ yields
	\begin{eqnarray}\label{121610}
	\|\left(\partial_tu^2\right)_R\|_{\widetilde\rho_1}\lesssim  R^{\fc{\widetilde\rho_1-\rho}{\widetilde\rho_1}}\|\partial_tu^2\|_{L^\rho}^{\fc{\rho}{\widetilde\rho_1}}.
	\end{eqnarray}
	It now  follows from (\ref{121610}) and  H\"older's inequality in time that
	\begin{equation}\label{11172}
	\|\left(\partial_tu^2\right)_R\|_{L^\gamma\left(I,L^{\widetilde\rho_1}\right)}  
 \lesssim T^{\fc{\widetilde\rho_1-\rho}{\gamma\widetilde\rho_1}}R^{\fc{\widetilde\rho_1-\rho}{\widetilde\rho_1}}\|\partial_tu^2\|_{L^{\gamma}\left(I,L^\rho\right)}^{\fc{\rho}{\widetilde\rho_1}}.
	\end{equation}
	This inequality together with (\ref{11302}) and H\"older's inequality yields
	\begin{eqnarray*}
		&&\left\|\left(\partial_tu^2\right)_RF(u^1,u^2)\right\|_{L^{\gamma'} \left(I,L^{\rho'}\right)}\notag\\
		&\lesssim& T^{1-\fc{\alpha+2}\gamma}\|\left(\partial_tu^2\right)_R\|_{L^\gamma\left(I,L^{\widetilde\rho_1}\right)}\|u^1-u^2\|_{L^\gamma\left(I,L^\rho\right)}^\alpha\\
		&\lesssim& T^{1-\fc{\alpha+2}\gamma+\fc{\widetilde\rho_1-\rho}{\gamma\widetilde\rho_1}}R^{\fc{\widetilde\rho_1-\rho}{\widetilde\rho_1}}\|\partial_tu^2\|_{L^{\gamma}\left(I,L^\rho\right)}^{\fc{\rho}{\widetilde\rho_1}}\|u^1-u^2\|_{L^\gamma\left(I,L^\rho\right)}^\alpha.
	\end{eqnarray*}

	We turn now to the case $\alpha>1$. Let $\widetilde\rho_2$ be given by the equation $\fc1{\rho'}=(\alpha-1)\left(\fc1\rho-\fc4N\right)+\fc1\rho+\fc1{\widetilde\rho_2}$. Then it is easy to check that $\rho<\widetilde \rho_2<\infty$.  Using the same method as that used to derive (\ref{11172}), we obtain
	\begin{eqnarray}\label{11202}
	\|\left(\partial_tu^2\right)_R\|_{L^\gamma\left(I,L^{\widetilde\rho_2}\right)}&\lesssim&T^{\fc{\widetilde\rho_2-\rho}{\gamma\widetilde\rho_2}}R^{\fc{\widetilde\rho_2-\rho}{\widetilde\rho_2}}\|\partial_tu^2\|_{L^{\gamma}\left(I,L^\rho\right)}^{\fc{\rho}{\widetilde\rho_2}}.
	\end{eqnarray}
	It now  follows from (\ref{11302}), (\ref{11202}), H\"older's inequality and Sobolev's embedding  $H^{4,\rho}\hookrightarrow L^\frac{\rho\alpha}{\rho-2}$ that
	\begin{eqnarray}\label{11203}
	&&\left\|\left(\partial_tu^2\right)_RF(u^1,u^2)\right\|_{L^{\gamma'} \left(I,L^{\rho'}\right)}\notag\\
	&\lesssim& T^{1-\fc{\alpha+2}\gamma}\|\partial_tu^2\|_{L^\gamma \left(I,L^{\widetilde\rho_2}\right)}\left(\|u^1\|_{L^\gamma \left(I,H^{4,\rho}\right)}^{\alpha-1}+\|u^2\|_{L^\gamma \left(I,H^{4,\rho}\right)}^{\alpha-1}\right)\|u^1-u^2\|_{L^\gamma \left(I,L^\rho\right)}\notag\\
	&\lesssim& T^{1-\fc{\alpha+2}\gamma+\fc{\widetilde\rho_2-\rho}{\gamma\widetilde\rho_2}}R^{\fc{\widetilde\rho_2-\rho}{\widetilde\rho_2}}\|\partial_tu^2\|_{L^{\gamma}\left(I,L^\rho\right)}^{\fc{\rho}{\widetilde\rho_2}}\left(\|u^1\|_{L^\gamma \left(I,H^{4,\rho}\right)}^{\alpha-1}+\|u^2\|_{L^\gamma \left(I,H^{4,\rho}\right)}^{\alpha-1}\right)\notag\\
	&&\qquad\times\|u^1-u^2\|_{L^\gamma \left(I,L^\rho\right)}.
	\end{eqnarray}
	
	This completes the proof of Lemma \ref{l1}.
\end{proof}
 In the process of the fixed-point argument, one first obtains estimates of $\partial_tu$
by time differentiation. Next, one obtains estimates of $\Delta^2u$ by using the equation
and estimates of $\partial_tu$. Thus, one needs to estimate $\left|u\right|^{\alpha}u$. To this end, we establish the following two Lemmas.
\begin{lemma}\label{l3}
	Assume that $u\in L^\gamma \left(I,L^{q_0}\right)\cap H^{1,\gamma}\left(I,L^\rho\right)$ with $u(0)=\phi\in H^4$, where $q_0$ is defined in (\ref{12171}). Then we have  $|u|^\alpha u\in C\left(I,L^2\right)$ and
	\begin{eqnarray}\label{12174}
	\left\|\left|u\right|^{\alpha}u\right\|_{L^\infty \left(I,L^2\right)}\lesssim \left\|\phi\right\|_{H^4}^{\alpha+1}+T^{\left(\alpha+1\right)\left(\frac{1}{\alpha+2}-\frac{1}{\gamma }\right)}\left\|u\right\|_{L^\gamma \left(I,L^{q_0}\right)}^{\frac{\left(\alpha+1\right)^2}{\alpha+2}}\left\|\partial_tu\right\|_{L^\gamma \left(I,L^\rho\right)}^{\frac{\alpha+1}{\alpha+2}}.
	\end{eqnarray}
	
	\begin{proof}
		We first show that $|u|^\alpha u\in C\left(I,L^2\right)$. Since $$\left||u|^\alpha u-|v|^\alpha v\right|\lesssim \left||u|^{\alpha+1} u-|v|^{\alpha+1} v\right|^{\fc{\alpha+1}{\alpha+2}}$$ (see Lemma 2.3 in \cite{Ca4}), it suffices to prove  that  $|u|^{\alpha+1} u\in C\left(I,L^{\fc{2(\alpha+1)}{\alpha+2}}\right)$.
		Since $\frac{\alpha+2}{2\left(\alpha+1\right)}=\frac{\alpha+1}{q_0}+\frac{1}{\rho}$ by  (\ref{12171}), it follows from H\"older's inequality that
		\begin{eqnarray}\label{1211}
		\left\|\partial_t\left(|u|^{\alpha+1}u\right)\right\|_{L^1 \left(I,L^{\fc{2(\alpha+1)}{\alpha+2}}\right)}&\lesssim & T^{1-\frac{\alpha+2}{\gamma }}\left\|u\right\|_{L^\gamma \left(I,L^{q_0}\right)}^{\alpha+1}\left\|\partial_tu\right\|_{L^\gamma \left(I,L^\rho\right)},
		\end{eqnarray}
		which implies that $\partial_t \left(\left|u\right|^{\alpha+1}u\right)\in L^1 \left(I,L^{\frac{2\left(\alpha+1\right)}{\alpha+2}}\right)$. On the other hand,  since  $\left|\phi\right|^{\alpha+1}\phi\in L^{\frac{2\left(\alpha+1\right)}{\alpha+2}}$ by Sobolev's embedding $H^4\hookrightarrow L^{2\left(\alpha+1\right)}$, we can apply (\ref{1211}) and the  Fundamental Theorem of Calculus
		\begin{eqnarray}\label{12175}
		\left|u\right|^{\alpha+1}u=\left|\phi\right|^{\alpha+1}\phi+\int_{0}^{t}\partial_s \left(\left|u\right|^{\alpha+1}u\left(s\right)\right)\mathrm{d}s
		\end{eqnarray}
		to  obtain $|u|^{\alpha+1} u\in C\left(I,L^{\fc{2(\alpha+1)}{\alpha+2}}\right)$, which in turn implies that $|u|^\alpha u\in C\left(I,L^2\right)$.
		
		Estimate (\ref{12174}) is now a consequence of  (\ref{1211}), (\ref{12175}) and the inequality $$\left\|\left|u\right|^{\alpha}u\right\|_{L^\infty \left(I,L^2\right)}\lesssim \left\|\left|u\right|^{\alpha+1}u\right\|_{L^\infty \left(I,L^{\frac{2\left(\alpha+1\right)}{\alpha+2}}\right)}^{\frac{\alpha+1}{\alpha+2}}.$$	\end{proof}
\end{lemma}

\begin{lemma}\label{l4}
	Assume $u\in L^\gamma \left(I,L^{q_0}\right)\cap L^\gamma \left(I,L^{p_0}\right)\cap H^{1,\gamma }\left(I,L^\rho\right)$ with $u(0)=\phi\in H^4$, where $q_0,p_0$ are  defined in (\ref{12171}) and (\ref{12201}). Then we have
	\begin{eqnarray}
	\left\|\left|u\right|^{\alpha}u\right\|_{L^\gamma \left(I,L^\rho\right)}&\lesssim& T^{\alpha \left(\frac{1}{\alpha+2}-\frac{1}{\gamma}\right)}\left(\left\|u\right\|_{L^\gamma \left(I,L^{q_0}\right)\cap L^\gamma \left(I,L^{p_0}\right)}^{\alpha+1}+\left\|\partial_tu\right\|_{L^\gamma \left(I,L^\rho\right)}^{\alpha+1}\right)\notag\\
	&&+\left\|\phi\right\|_{H^4}^\alpha F\left(\phi,T\right)\notag.
	\end{eqnarray}
\end{lemma}
\begin{proof}
	We assume first that  $\phi=0$. Since $\frac{1}{\rho}=\frac{\alpha}{2(\alpha+1)}+\frac{1}{p_0}$ by (\ref{12201}), it follows from  H\"older's inequality  that
	\begin{eqnarray}\label{11215}
	\left\|\left|u\right|^{\alpha}u\right\|_{L^\gamma \left(I,L^\rho\right)}&\lesssim& \left\|\left\|u\right\|^\alpha_{2\left(\alpha+1\right)}\left\|u\right\|_{p_0}\right\|_{L^\gamma \left(I\right)}\notag\\
	&\lesssim& \left\|u\right\|^\alpha_{L^\infty \left(I,L^{2\left(\alpha+1\right)}\right)}\left\|u\right\|_{L^\gamma \left(I,L^{p_0}\right)}.
	\end{eqnarray}
	Our next step is to  estimate $\left\|u\right\|^\alpha_{L^\infty \left(I,L^{2\left(\alpha+1\right)}\right)}$. Since $u(0)=\phi=0$, we see that
	\begin{eqnarray}
	\left\|u\right\|_{2(\alpha+1)}^{\alpha+2}&=&\left\|\left|u\right|^{\alpha+1}u(t)\right\|_{\frac{2\left(\alpha+1\right)}{\alpha+2}}\notag\\
	&=&\left\|\int_{0}^{t}\partial_s \left(\left|u\right|^{\alpha+1}u(s)\right)\mathrm{d}s\right\|_{\frac{2\left(\alpha+1\right)}{\alpha+2}}\notag\\
	&\lesssim& \left\|\partial_t \left(\left|u\right|^{\alpha+1}u\right)\right\|_{L^1 \left(I,L^{\frac{2\left(\alpha+1\right)}{\alpha+2}}\right)}.\notag
	\end{eqnarray}
	This inequality together with (\ref{1211}) yields
	\begin{eqnarray}\label{1212}
	\left\|u\right\|^\alpha_{L^\infty \left(I,L^{2\left(\alpha+1\right)}\right)}\lesssim T^{\alpha\left(\frac{1}{\alpha+2}-\frac{1}{\gamma }\right)}\left\|u\right\|_{L^\gamma \left(I,L^{q_0}\right)\cap L^\gamma \left(I,L^{p_0}\right)}^{\frac{\alpha \left(\alpha+1\right)}{\alpha+2}}\left\|\partial_tu\right\|_{L^\gamma \left(I,L^\rho\right)}^{\frac{\alpha}{\alpha+2}}.
	\end{eqnarray}
	It now follows from  (\ref{11215}), (\ref{1212}) and Young's inequality $\left|x^{\frac{\alpha^2+2\alpha+2}{\alpha+2}}y^\frac{\alpha}{\alpha+2}\right|\lesssim \left|x\right|^{\alpha+1}+\left|y\right|^{\alpha+1}$ that   Lemma \ref{l4} holds in the case $\phi=0$.
	
	We now consider the  case $\phi\neq0$. Let $v(t)=u(t)-e^{it(\Delta^2+\mu\Delta)}\phi$. It follows from the elementary inequality $\left|x+y\right|^{\alpha+1}\lesssim \left|x\right|^{\alpha+1}+\left|y\right|^{\alpha+1}$ that
	\begin{eqnarray}\label{11216}
	&&\left\|\left|u\right|^{\alpha}u\right\|_{L^\gamma \left(I,L^\rho\right)}=\left\|u\right\|_{L^{\left(\alpha+1\right)\gamma } \left(I,L^{\left(\alpha+1\right)\rho}\right)}^{\alpha+1}\notag\\
	&\lesssim &\left\|v\right\|_{L^{\left(\alpha+1\right)\gamma } \left(I,L^{\left(\alpha+1\right)\rho}\right)}^{\alpha+1}+\left\|e^{it(\Delta^2+\mu\Delta)}\phi\right\|_{L^{\left(\alpha+1\right)\gamma } \left(I,L^{\left(\alpha+1\right)\rho}\right)}^{\alpha+1}\notag\\
	&\lesssim &\left\|\left|v\right|^{\alpha}v\right\|_{L^\gamma \left(I,L^\rho\right)}+\left\|\left|e^{it(\Delta^2+\mu\Delta)}\phi\right|^{\alpha}e^{it(\Delta^2+\mu\Delta)}\phi\right\|_{L^\gamma \left(I, L^\rho\right)}.
	\end{eqnarray}
	Since $v(0)=0$, we can apply the results established in the previous case to  obtain
	\begin{eqnarray}\label{12132}
	&&\left\|\left|v\right|^{\alpha}v\right\|_{L^\gamma \left(I,L^\rho\right)}\notag\\
	&\lesssim& T^{\alpha \left(\frac{1}{\alpha+2}-\frac{1}{\gamma }\right)}\left(\left\|v\right\|_{L^\gamma \left(I,L^{q_0}\right)\cap L^\gamma \left(I,L^{p_0}\right)}^{\alpha+1}+\left\|\partial_tv\right\|_{L^\gamma \left(I,L^\rho\right)}^{\alpha+1}\right).
	\end{eqnarray}
	Moreover, since  $v(t)=u(t)-e^{it(\Delta^2+\mu\Delta)}\phi$,  we deduce from (\ref{12132}) and the elementary inequality $\left|x+y\right|^{\alpha+1}\lesssim \left|x\right|^{\alpha+1}+\left|y\right|^{\alpha+1}$ again  that
	\begin{eqnarray}\label{11218}
	\left\|\left|v\right|^{\alpha}v\right\|_{L^\gamma \left(I,L^\rho\right)}
	&\lesssim& T^{\alpha \left(\frac{1}{\alpha+2}-\frac{1}{\gamma }\right)}\left(\left\|u\right\|_{L^\gamma \left(I,L^{q_0}\right)\cap L^\gamma \left(I,L^{p_0}\right)}^{\alpha+1}\right. \notag\\
	&&\left.\qquad+\left\|\partial_tu\right\|_{L^\gamma \left(I,L^\rho\right)}^{\alpha+1}+F\left(\phi,T\right)^{\alpha+1}\right),
	\end{eqnarray}
	where we used the embedding (\ref{a1}) to control $L^\gamma \left(I,L^{q_0}\right)\cap L^\gamma \left(I,L^{p_0}\right)$ norm of $e^{it(\Delta^2+\mu\Delta)}\phi$. On the other hand, applying the same method as that used to derive (\ref{11215}), we obtain
	\begin{eqnarray}\label{11217}
	&&\left\|\left|e^{it(\Delta^2+\mu\Delta)}\phi\right|^{\alpha}e^{it(\Delta^2+\mu\Delta)}\phi\right\|_{L^\gamma \left(I,L^\rho\right)}\notag\\
	&\lesssim& \left\|e^{it(\Delta^2+\mu\Delta)}\phi\right\|_{L^\infty \left(I,L^{2\left(\alpha+1\right)}\right)}^\alpha \left\|e^{it(\Delta^2+\mu\Delta)}\phi\right\|_{L^\gamma \left(I,L^{p_0}\right)}\notag\\
	&\lesssim& \left\|\phi\right\|_{H^4}^\alpha F\left(\phi,T\right),
	\end{eqnarray}
	where  we  used the embedding  (\ref{a1}) and Sobolev's embedding $H^4\hookrightarrow L^{2\left(\alpha+1\right)}$ in the last inequality.

	Combining (\ref{11216}), (\ref{11218}) and (\ref{11217}), we obtain Lemma \ref{l4} in the case $\phi\neq0$.
\end{proof}
The next Lemma is the key ingredient in our proof of the local existence and continuous dependence.
\begin{lemma}\label{l2}
	Assume  that $u^1,u^2\in L^\infty \left(I,H^4\right)\cap L^\gamma \left(I,H^{4,\rho}\right)\cap H^{1,\gamma }\left(I,L^\rho\right)$, then we have
	\begin{eqnarray}\label{111620}
	&&\left\||u^1|^\alpha u^1-|u^2|^\alpha u^2\right\|_{L^\infty \left(I,L^2\right)}\notag\\
	&\lesssim&T^{\frac{1}{\alpha+2}-\frac{1}{\gamma }}\left(\|u^1\|_{L^\infty \left(I,H^4\right)}^\alpha+\|u^2\|_{L^\infty \left(I,H^4\right)}^\alpha\right)\notag\\
	&&\times\left\|u^1-u^2\right\|_{L^\gamma \left(I,H^{4,\rho}\right)}^{\frac{ \alpha+1}{\alpha+2}}\left\|\partial_tu^1-\partial_tu^2\right\|_{L^\gamma \left(I,L^\rho\right)}^{\frac{1}{\alpha+2}},
	\end{eqnarray}
	and
	\begin{eqnarray}\label{111630}
	&&\left\||u^1|^\alpha u^1-|u^2|^\alpha u^2\right\|_{L^2\left(I,L^{\fc{2N}{N-4}}\right)}\notag\\
	&\lesssim& T^{\fc1{\alpha+2}-\fc1\gamma}\left(\|u^1\|_{L^\infty \left(I,H^4\right)\cap L^\gamma \left(I,B^4_{\rho,2}\right)}^\alpha+\|u^2\|_{L^\infty \left(I,H^4\right)\cap L^\gamma \left(I,B^4_{\rho,2}\right)}^\alpha\right)\notag\\
	&&\times\|u^1-u^2\|_{L^\gamma \left(I,H^{4,\rho}\right)}.
	\end{eqnarray}
\end{lemma}
\begin{proof}
	We first prove (\ref{111620}).
	It follows from  H\"older's inequality and Sobolev's embedding $H^4 \hookrightarrow L^{2\left(\alpha+1\right)}$ that
	\begin{eqnarray}
	&&\left\|\left|u^1\right|^{\alpha}u^1-\left|u^2\right|^{\alpha}u^2\right\|_{L^\infty \left(I,L^2\right)}\notag\\
	&\lesssim &\left(\left\|u^1\right\|^\alpha_{L^\infty \left(I,H^4\right)}+\left\|u^2\right\|_{L^\infty \left(I,H^4\right)}^\alpha\right)\left\|u^1-u^2\right\|_{L^\infty \left(I,L^{2\left(\alpha+1\right)}\right)}.
	\end{eqnarray}
	This inequality together with (\ref{1212}) and the embedding (\ref{a1}) yields (\ref{111620}).

	We now prove (\ref{111630}). Since $\rho=\frac{N\left(\alpha+2\right)}{N+4\alpha}$ and $\alpha\le\fc8{N-8}$, we have by direct calculation
	\begin{eqnarray*}
		&&\fc{N-4}{2N}-\fc\alpha2\left(\fc12-\fc4N\right)-\left(\fc\alpha2+1\right)\left(\fc1\rho-\fc4N\right)=\fc{8-(N-8)\alpha}{4N}\ge0,
	\end{eqnarray*}
	and
	\begin{eqnarray*}
		\frac{N-4}{2N}-\frac{\alpha}{2}\cdot \frac{1}{2}-\left(\frac{\alpha}{2}+1\right)\frac{1}{\rho}=-\frac{\alpha}{4}-\frac{2\left(\alpha+1\right)}{N}<0.
	\end{eqnarray*}
	So that the straight line of points $(x,y)$ satisfying $y=-\frac{\alpha}{\alpha+2}x+\frac{N-4}{N\left(\alpha+2\right)}$ passes through the square (including edges) with vertices $(\frac{1}{2}-\frac{4}{N},\frac{1}{\rho}-\frac{4}{N}),$ $(\frac{1}{2},\frac{1}{\rho}-\frac{4}{N})$ $ (\frac{1}{2}-\frac{4}{N},\frac{1}{\rho})$ and $(\frac{1}{2},\frac{1}{\rho})$.
	This implies that there exist $p_1, p_2$ satisfying $\frac12\ge\frac{1}{p_1}\ge\frac{1}{2}-\frac{4}{N},\frac{1}{\rho}\ge\frac{1}{p_2}\ge\frac{1}{\rho}-\frac{4}{N}$ and  $\frac{N-4}{2N}=\frac{\alpha}{2}\frac{1}{p_1}+\left(\frac\alpha2+1\right)\frac{1}{p_2}$.
	Hence we can deduce from H\"older's inequality and Sobolev's embedding $H^4 \hookrightarrow L^{p_1}, H^{4,\rho}\hookrightarrow L^{p_2}$ that

	\begin{eqnarray*}
		\left\|\left|u^1\right|^{\alpha}|u^1-u^2|\right\|_{L^{\fc{2N}{N-4}}}&\lesssim& \left\|u^1\right\|_{p_1}^\frac{\alpha}{2}\left\|u^1\right\|_{p_2}^{\frac{\alpha}{2}}\left\|u^1-u^2\right\|_{p_2}\notag\\
		&\lesssim &\left\|u^1\right\|_{H^4}^\frac{\alpha}{2}\left\|u^1\right\|_{B^4_{\rho,2}}^\frac{\alpha}{2}\left\|u^1-u^2\right\|_{B^4_{\rho,2}}
	\end{eqnarray*}
	This inequality together with H\"older's inequality in time gives
	\begin{eqnarray}\label{12242}
	&&\left\|\left|u^1\right|^{\alpha}\left(u^1-u^2\right)\right\|_{L^2\left(I,L^{\frac{2N}{N-2}}\right)}\notag\\
	&\lesssim& T^{\frac{1}{\alpha+2}-\frac{1}{\gamma }}\left\|u^1\right\|_{L^\infty \left(I,H^4\right)}^{\frac{\alpha}{2}}\left\|u^1\right\|_{L^\gamma \left(I,B^4_{\rho,2}\right)}^\frac{\alpha}{2}\left\|u^1-u^2\right\|_{L^\gamma \left(I,B^4_{\rho,2}\right)}.
	\end{eqnarray}
	Similarly,
	\begin{eqnarray}\label{12243}
		&&\left\|\left|u^2\right|^{\alpha}\left(u^1-u^2\right)\right\|_{L^2\left(I,L^{\frac{2N}{N-2}}\right)}\notag\\
	&\lesssim& T^{\frac{1}{\alpha+2}-\frac{1}{\gamma }}\left\|u^2\right\|_{L^\infty \left(I,H^4\right)}^{\frac{\alpha}{2}}\left\|u^2\right\|_{L^\gamma \left(I,B^4_{\rho,2}\right)}^\frac{\alpha}{2}\left\|u^1-u^2\right\|_{L^\gamma \left(I,B^4_{\rho,2}\right)}.
	\end{eqnarray}
	Combining (\ref{12242}) and (\ref{12243}), we obtain (\ref{111630}).
\end{proof}
\section{Proof of Theorem \ref{T1.00} in the case $N\ge9$}\label{s4}
In this section, we  use a contraction mapping argument based on the  nonlinear estimates established in the previous Section to prove Theorem \ref{T1.00}. The  case $1\leq N\leq8$ is  more simpler to deal with, since Sobolev's embedding theorem $H^4(\R^N)\hookrightarrow L^p(\R^N)$ holds for any $2\leq p<\infty$. We append its proof in the last section. In this section, we consider only  the case $N\geq9$.

We start with the following proposition.
\begin{proposition}\label{P4.1}
	Assume  $\lambda\in\mathbb{C},\mu=\pm1$ or $0, N\ge9,0<\alpha<\frac{8}{N-8},\phi\in H^4$ and $M\ge 2C_1\left(\left\|\phi\right\|_{H^4}^{2\alpha}+1\right)\left\|\phi\right\|_{H^4}^\alpha+1$, where $C_1$ is the constant in (\ref{1219}). Then for any $0<T\le c\left(M\right)$, where $c(M)$ is the constant in (\ref{T}),  the Cauchy problem (\ref{NLS})-(\ref{12241}) admits a unique solution $u\in C\left([0,T],H^4\right)\cap L^\gamma \left([0,T],H^{4,\rho}\right)$ with
		$
	\left\|u\right\|_{H^{1,\gamma  }\left([0,T],L^\rho\right)\cap L^\gamma  \left([0,T],H^{4,\rho}\right)}\le M.
$
	  Moreover, we have the further regularity:
	 $
	  	u,u_{t}, \Delta^2 u \in C\left([0, T], L^{2}\left(\mathbb{R}^{N}\right)\right) \cap L^{q}\left((0, T), L^{r}\left(\mathbb{R}^{N}\right)\right)
	  $ for every biharmonic admissible pair $(q,r)\in\Lambda_b$.

\end{proposition}
\begin{proof}To begin with, we fix $c\left(M\right)>0$ such that
	\begin{eqnarray}\label{T}
	\left(C_1+C_2\right)c\left(M\right)^{\alpha \left(\frac{1}{\alpha+2}-\frac{1}{\gamma}\right)}M^\alpha=\frac12,
	\end{eqnarray}
	where $ C_2$ are the constant (\ref{12110}).

	Set $I=[0,T]$ and consider the metric space
	\begin{eqnarray}\label{XTM}
	X_{T,M}=\left\{u\in H^{1,\gamma } \left(I,L^\rho\right)\cap L^\gamma \left(I,H^{4,\rho}\right): \left\|u\right\|_{H^{1,\gamma } \left(I,L^\rho\right)\cap L^\gamma \left(I,H^{4,\rho}\right)}\le M\right\}
	\end{eqnarray}
	It follows that $X_{T,M}$ is a complete metric space when equipped with the distance
	\begin{eqnarray}\label{1218}
	d(u,v)=\left\|u-v\right\|_{L^\gamma \left(I,L^\rho\right)}.
	\end{eqnarray}
	To obtain a solution $u\in X_{T,M}$ to the equation (\ref{NLS}), it suffices to show that the map  $S$, defined in (\ref{12113}), is a contraction on the space $X_{T,M}$.
	
	We first show that $S$ maps $X_{T,M}$ into itself. By  (\ref{12111}) and (\ref{1214}), we have for any $u\in X_{T,M}$
	\begin{eqnarray}\label{7.15}
	\|Su\|_{H^{1,\gamma}(I,L^\rho)}\lesssim \left\|\phi\right\|_{H^4}+\left\|\phi\right\|_{H^4}^{\alpha+1}+T^{1-\frac{\alpha+2}{\gamma }}\left\|u\right\|_{L^\gamma \left(I,H^{4,\rho}\right)}^\alpha \left\|u\right\|_{H^{1,\gamma } \left(I,L^\rho\right)}.
	\end{eqnarray}
	Our next step is to estimate  $\|Su\|_{L^\gamma(I,H^{4,\rho})}$. It follows from the equation (\ref{1217}) that, for every $u\in X_{T,M}$,
	\begin{equation}\label{7.16}
	\|\Delta^2(Su)\|_{L^\gamma(I,L^\rho)}\leq\|\Delta (Su)\|_{L^\gamma(I,L^\rho)}+\|\partial_t (Su)\|_{L^\gamma(I,L^\rho)}+\||u|^\alpha u\|_{L^\gamma(I,L^\rho)}.
	\end{equation}
	Moreover, it follows from H\"older's inequality and Gagliardo-Nirenberg's inequality that
	\begin{eqnarray}\label{111213}
	\|\Delta\left(Su\right)\|_{L^\gamma\left(I,L^\rho\right)}&\le& C\left\|\|Su\|_{L^\rho}^\fc12\|\Delta^2(Su)\|_{L^\rho}^\fc12\right\|_{L^\gamma(I)}\notag\\
	&\le& C\|Su\|_{L^\gamma\left(I,L^\rho\right)}^\fc12\|\Delta^2\left(Su\right)\|_{L^\gamma\left(I,L^\rho\right)}^\fc12\notag\\
	&\le&\fc12\|\Delta^2\left(Su\right)\|_{L^\gamma\left(I,L^\rho\right)}+\fc12C^2\|Su\|_{L^\gamma\left(I,L^\rho\right)},
	\end{eqnarray}
	where we  used Cauchy-Schwartz's inequality in the last step. Estimates (\ref{7.16}) and (\ref{111213}) imply that
	\begin{eqnarray}\label{ab}
		\|\Delta^2(Su)\|_{L^\gamma(I,L^\rho)}\lesssim \|\partial_t (Su)\|_{L^\gamma(I,L^\rho)}+\|Su\|_{L^\gamma(I,L^\rho)}+\||u|^\alpha u\|_{L^\gamma(I,L^\rho)}.
	\end{eqnarray}
	On the other hand, applying  Lemma \ref{l4}, (\ref{12111}) and  (\ref{a1}), we conclude that
	\begin{eqnarray}\label{11211}
	\left\|\left|u\right|^{\alpha}u\right\|_{L^\gamma \left(I,L^\rho\right)}&\lesssim& T^{\alpha \left(\frac{1}{\alpha+2}-\frac{1}{\gamma}\right)}\left(\left\|u\right\|_{L^\gamma \left(I,H^{4,\rho}\right)}^{\alpha+1}+\left\|\partial_tu\right\|_{L^\gamma \left(I,L^\rho\right)}^{\alpha+1}\right)\notag\\
	&&+\left\|\phi\right\|_{H^4}^{\alpha+1}+\left\|\phi\right\|_{H^4}^{2\alpha+1}.
	\end{eqnarray}
		It now follows from (\ref{7.15}), (\ref{ab}) and (\ref{11211}) that, for $u\in X_{T,M}$
	\begin{eqnarray}\label{1219}
	&&\|Su\|_{H^{1,\gamma}(I,L^\rho)\cap L^\gamma(I,H^{4,\rho})}\notag\\
	&\le& C_1\left(\left\|\phi\right\|_{H^4}^{2\alpha}+1\right)\left\|\phi\right\|_{H^4}+C_1T^{\alpha \left(\frac{1}{\alpha+2}-\frac{1}{\gamma}\right)}M^{\alpha+1}.
	\end{eqnarray}
	Since $0<T\le c\left(M\right)$, we can apply (\ref{T}) and (\ref{1219}) to obtain
	\begin{eqnarray*}
		\|Su\|_{H^{1,\gamma}(I,L^\rho)\cap L^\gamma(I,H^{4,\rho})}\le M.
	\end{eqnarray*}
	
	Our next aim is the desired Lipschitz property of $S$ with respect to the metric $d$ defined in (\ref{1218}).  Given  $u,v\in X_{T,M}$, we can apply the same method as that used to derive (\ref{1216}) to obtain
	\begin{eqnarray}\label{12110}
	d(Su,Sv)
	&\lesssim& T^{1-\frac{\alpha+2}\gamma}\left(\|u\|_{L^\gamma(I,H^{4,\rho})}^\alpha+\|v\|_{L^\gamma(I,H^{4,\rho})}^\alpha\right)\|u-v\|_{L^{\gamma}(I,L^{\rho})}\notag\\
	&\le& C_2T^{1-\frac{\alpha+2}\gamma}M^\alpha d(u,v)\le\frac12 d\left(u,v\right).
	\end{eqnarray}
	
	So we prove that $S$ is a contraction on the space $X_{T,M}$ for $0<T\le c\left(M\right)$. Using Banach's fixed-point argument, we deduce that there exists a unique solution $u\in L^\infty \left([0,T],H^4\right)\cap L^\gamma \left([0,T],H^{4,\rho}\right)$ to the Cauchy problem (\ref{NLS})-(\ref{12241}).

	We now prove some further regularity properties.
	
	Firstly, we prove that  $u \in H^{1,q}\left((0, T), L^{r}\left(\mathbb{R}^{N}\right)\right)$ for every biharmonic admissible pair $(q, r)\in\Lambda_b$ and $u, \partial_tu \in C\left([0, T], L^{2}\left(\mathbb{R}^{N}\right)\right)$. Note
	that we have (see (\ref{12113}) and (\ref{NLS2}))
	\begin{eqnarray}\label{12213}
	u(t)=e^{it(\Delta^2+\mu\Delta)}\phi+i\lambda \int_{0}^{t}e^{i\left(t-s\right)(\Delta^2+\mu\Delta)}[\left|u\right|^{\alpha}u]\left(s\right)\mathrm{d}s,
	\end{eqnarray}
	and
	$$
	\partial_tu=i e^{i t (\Delta^2+\mu\Delta)}\left[\left(\Delta^2+\mu\Delta\right) \phi-\lambda|\phi|^{\alpha} \phi\right]-i \lambda \int_{0}^{t} e^{i(t-s) \left(\Delta^2+\mu\Delta\right)} \partial_{s}\left[|u|^{\alpha} u\right](s) d s.
	$$
	Since $\phi,\left(\Delta^2+\mu\Delta\right) \phi-\lambda|\phi|^{\alpha} \phi \in L^{2}\left(\mathbb{R}^{N}\right)$ and $\left[|u|^{\alpha} u\right] \in H^{1,\gamma^{\prime}}\left((0, T), L^{\rho^{\prime}}\left(\mathbb{R}^{N}\right)\right)$  (see the proof of (\ref{1215})), we deduce  from Strichartz's estimates (\ref{SZ}) that $u \in H^{1,q}\left((0, T), L^{r}\left(\mathbb{R}^{N}\right)\right)$ for every biharmonic admissible pair $(q, r)\in\Lambda_b$ and $u, \partial_tu \in C\left([0, T], L^{2}\left(\mathbb{R}^{N}\right)\right)$.
	
	Next we prove  that $\Delta^2 u \in C\left([0, T], L^{2}\left(\mathbb{R}^{N}\right)\right)$. For any $t_1,t_2\in[0,T]$, we have by using the equation (\ref{NLS})
	\begin{eqnarray}\label{11182}
	\left\|\Delta^2u(t_1)-\Delta^2u(t_2)\right\|_2&\lesssim&\|\partial_tu(t_1)-\partial_tu(t_2)\|_2+\left\|\Delta u(t_1)-\Delta u(t_2)\right\|_2\notag\\
	&&+\left\||u|^\alpha u(t_1)-|u|^\alpha u(t_2)\right\|_2
	\end{eqnarray}
	Using the same method as that used to derive (\ref{ab}), we deduce that
	\begin{eqnarray}\label{1222}
	\left\|\Delta^2u(t_1)-\Delta^2u(t_2)\right\|_2&\lesssim&\|\partial_tu(t_1)-\partial_tu(t_2)\|_2+\left\| u(t_1)- u(t_2)\right\|_2\notag\\
	&&+\left\||u|^\alpha u(t_1)-|u|^\alpha u(t_2)\right\|_2.
	\end{eqnarray}
	Since   $u,\partial_tu$ and $|u|^{\alpha} u \in C\left([0, T], L^{2}\left(\mathbb{R}^{N}\right)\right)$ (see Lemma \ref{l3} for the continuity of $|u|^\alpha u$), we deduce from (\ref{1222})   that $\Delta^2 u \in C\left([0, T], L^{2}\left(\mathbb{R}^{N}\right)\right)$. In particular, we have $u\in L^\infty \left([0,T],H^4\left(\R^N\right)\right)$.
	
	Our final step is to show that $\Delta^2 u \in L^{q}\left((0, T), L^{r}\left(\mathbb{R}^{N}\right)\right)$ for every admissible pair $(q,r)\in\Lambda_b$. Similarly to (\ref{ab}), we have
	\begin{eqnarray}\label{12203}
	\left\|\Delta^2 u\right\|_{L^q \left(\left(0,T\right),L^r\right)}\lesssim \left\| u\right\|_{H^{1,q} \left(\left(0,T\right),L^r\right)}+\left\|\left|u\right|^{\alpha}u\right\|_{L^q \left(\left(0,T\right),L^r\right)}
	\end{eqnarray} On the other hand, since $ u \in L^{\gamma}\left((0, T), H^{4,\rho}\left(\mathbb{R}^{N}\right)\right) \cap L^{\infty}\left((0, T), H^{4}\left(\mathbb{R}^{N}\right)\right)\cap H^{1,\gamma }\left(\left(0,T\right),L^\rho \left(\R^N\right)\right),$ it follows from Lemma \ref{l2} that $|u|^{\alpha} u \in L^{2}((0, T), L^{\frac{2 N}{N-2}}$ $(\mathbb{R}^{N}))$ $\cap L^{\infty}\left((0, T), L^{2}\left(\mathbb{R}^{N}\right)\right)$. This together with  H\"older's inequality implies that  $|u|^{\alpha} u \in$ $L^{q}\left((0, T), L^{r}\left(\mathbb{R}^{N}\right)\right)$ for every admissible pair $(q,r)\in\Lambda_b$. Therefore, from   (\ref{12203}), we see that $\Delta^2 u \in L^{q}\left((0, T), L^{r}\left(\mathbb{R}^{N}\right)\right)$ for every admissible pair $(q,r)\in\Lambda_b$.
	
	The proof of Proposition \ref{P4.1} is now completed.
\end{proof}
\begin{proof}[\textbf{Proof of Theorem \ref{T1.00}in the case $N\ge9$}]
	We now resume the proof of Theorem \ref{T1.00}.  We consider only the positive time direction. A corresponding conclusion for  reverse direction follows similarly. Firstly, we deduce from Proposition \ref{P4.1} that there exist $T>0$ and a unique solution $u\in C\left([0,T],H^4\right)$ to the Cauchy problem (\ref{NLS})-(\ref{12241}) with  $u,u_{t}, \Delta^2 u \in$ $C\left([0, T], L^{2}\left(\mathbb{R}^{N}\right)\right)\cap L^q\left([0,T],L^r\left(\R^N\right)\right)$ for every biharmonic admissible pair $(q,r)\in\Lambda_b$. Using the uniqueness in  Appendix and the standard procedure as in Chapter 4 of \cite{Ca9}, we can extend $u$ to a maximal solution $u\in C\left([0,T_{\text{max}}),H^4\right)$ to the Cauchy problem (\ref{NLS})-(\ref{12241}) with  $u,u_{t}, \Delta^2 u \in$ $C\left([0, T_{\text{max}}), L^{2}\left(\mathbb{R}^{N}\right)\right) \cap L^{q}_{\text{loc}}\left((0, T_{\text{max}}), L^{r}\left(\mathbb{R}^{N}\right)\right)$ for every biharmonic admissible pair $(q,r)\in\Lambda_b$. Moreover, since the solution $u$ of  Proposition \ref{P4.1} is constructed on an interval depending on $\left\|\phi\right\|_{H^4}$, we deduce the blowup alternative (\ref{bl1}). 
	
	In the rest of this section, we prove the continuous dependence of the solution  map.

	For any $0<A<T_{\text{max}}(\phi)$, we set
	\begin{eqnarray}\label{M}
	M&=&4C_1\left(\left\|u\right\|_{L^\infty\left([0,A],H^4\right)}^{2\alpha}+1\right)\left\|u\right\|_{L^\infty\left([0,A],H^4\right)}^\alpha\notag\\
	&&+\left\|u\right\|_{L^\gamma \left([0,A],H^{4,\rho}\right)\cap H^{1,\gamma }\left([0,A],L^\rho\right)}+2
	\end{eqnarray}
	where  $C_1$ is the constant in (\ref{1219}). Moreover, we fix $T>0$ sufficiently small such that
	\begin{equation}\label{9171}
	T\le c(M), \left(C_3+C_4\right)T^{1-\frac{\alpha+2}{\gamma }}M^\alpha+C_5 T^{\frac{\alpha}{\alpha+1}\left(\frac{1}{\alpha+2}-\frac{1}{\gamma }\right)}\left(M^\alpha+M^{\frac{\alpha \left(2\alpha+1\right)}{\alpha+1}}\right)\le \frac{1}{2}
	\end{equation}
	where $c(M)$ is the constant in Proposition \ref{P4.1} and $C_3,C_4, C_5$ are the constants  in (\ref{12112}),  (\ref{12115}) and (\ref{121614}), respectively.
	
	Since  $\phi_n\rightarrow\phi$ in $H^4$, we know that there exists a positive $n_1$ such that $M\ge2C_1\left(\|\phi_n\|_{H^4}^{2\alpha}+1\right)\|\phi_n\|_{H^4}^\alpha+1$ for every $n\ge n_1$. Then we can deduce from Proposition \ref{P4.1} that for every $n\ge n_1$, the following equation
	\begin{equation}\label{1221}
	u_n=e^{it(\Delta^2+\mu\Delta)}\phi_n+i\lambda \int_{0}^{t}e^{i\left(t-s\right)(\Delta^2+\mu\Delta)}[\left|u_n\right|^{\alpha}u_n]\left(s\right)\mathrm{d}s
	\end{equation}
	admits a unique solution $u_n\in C(I,H^4)\cap  L^\gamma(I,B^{4}_{\rho,2})$ with $I=[0,T]$ and $u_n$ is uniformly bounded:
	\begin{equation}\label{M1}
	\|u_n\|_{L^\gamma  \left(I,H^{4,\rho}\right)\cap H^{1,\gamma }\left(I,L^\rho\right)} \leq M,\ \text{for } \forall\  n\ge n_1.
	\end{equation}
	The proof of the continuous dependence will proceed by a series of  Claims.
	
	\begin{claim}\label{c2}
	Given any $(q,r)\in\Lambda_b$, we have $\left\|u_n-u\right\|_{H^{1,q} \left(I,L^r\right)}\rightarrow0$ as $n\rightarrow\infty$.
	\end{claim}
	\begin{proof}
		We first show that
		\begin{eqnarray}\label{121611}
		\lim_{n\rightarrow \infty }\left\|u_n-u\right\|_{L^q \left(I,L^r\right)\cap L^\gamma \left(I,L^\rho\right)}=0.
		\end{eqnarray}
		Applying  the same method as that used to derive (\ref{1216}), we conclude that
		\begin{eqnarray}\label{12112}
		&&\left\|u_n-u\right\|_{L^q \left(I,L^r\right)\cap L^\gamma \left(I,L^\rho\right)}\notag\\
		&\lesssim&\left\|\phi_n-\phi\right\|_{L^2}+T^{1-\frac{\alpha+2}\gamma}\left(\|u_n\|_{L^\gamma(I,H^{4,\rho})}^\alpha+\|u\|_{L^\gamma(I,H^{4,\rho})}^\alpha\right)\|u_n-u\|_{L^{\gamma}(I,L^{\rho})}\notag\\
		&\le&C_3\left\|\phi_n-\phi\right\|_{L^2}+C_3T^{1-\fc{\alpha+2}\gamma}M^\alpha\left\|u_n-u\right\|_{L^q \left(I,L^r\right)},
		\end{eqnarray}
		where we used the boundedness of $u_n,u$ in   (\ref{M}) and  (\ref{M1}). This inequality together with the smallness of $T$ in (\ref{9171}) yields  (\ref{121611}).
		
		Next, it follows from  Lemma \ref{l1}, (\ref{121611}) and the boundedness of $u_n,u$ again, we have, for any $R>0$
		\begin{eqnarray*}
			&&\limsup_{n\rightarrow\infty}\|u_n-u\|_{H^{1,q}\left(I,L^r\right)\cap H^{1,\gamma }\left(I,L^\rho\right)}\notag\\
			&\lesssim&T^{1-\fc{\alpha+2}\gamma}M^\alpha \limsup_{n\rightarrow\infty}\|u_n-u\|_{H^{1,\gamma } \left(I,L^\rho\right)}+T^{1-\fc{\alpha+2}\gamma}\|\left(\partial_tu\right)^R\|_{L^\gamma \left(I,L^\rho\right)}M^\alpha.
		\end{eqnarray*}
		Note that $\lim_{R\rightarrow \infty }\|\left(\partial_tu\right)^R\|_{L^\gamma \left(I,L^\rho\right)}=0$ by the dominated convergence theorem, we can let $R\rightarrow\infty$ to obtain
		\begin{eqnarray}\label{12115}
		&&\limsup_{n\rightarrow\infty}\|u_n-u\|_{H^{1,q}\left(I,L^r\right)\cap H^{1,\gamma }\left(I,L^\rho\right)}\notag\\
		&\le& C_4T^{1-\fc{\alpha+2}\gamma}M^\alpha \limsup_{n\rightarrow\infty}\|\partial_tu_n-\partial_tu\|_{L^\gamma \left(I,L^\rho\right)}.
		\end{eqnarray}
		Since  $C_4T^{1-\fc{\alpha+2}\gamma}M^\alpha\le\fc12$ by (\ref{9171}), it follows from (\ref{12115}) that  $\limsup_{n\rightarrow\infty}\|u_n-u\|_{H^{1,q}\left(I,L^r\right)}=0$.
		
		This finishes the proof of Claim \ref{c2}.
	\end{proof}
	\begin{claim}\label{cc1}
		$	\left\|u_n-u\right\|_{L^\gamma   \left(I,H^{4,\rho}\right)}\rightarrow0$ as $n\rightarrow\infty$.
	\end{claim}
	\begin{proof}
		Using the same method as that used to derive (\ref{ab}), we obtain
		\begin{eqnarray}\label{1118200}
		\left\|u_n-u\right\|_{L^\gamma   \left(I,H^{4,\rho}\right)}
		\lesssim\|u_n-u\|_{H^{1,\gamma } \left(I,L^\rho\right)}+\left\||u_n|^\alpha u_n-|u|^\alpha u\right\|_{L^\gamma  \left(I,L^\rho\right)}.
		\end{eqnarray}
		Next, we  estimate  $\left\||u_n|^\alpha u_n-|u|^\alpha u\right\|_{L^\gamma  \left(I,L^\rho\right)}$. It follows from  H\"older's inequality that
		\begin{eqnarray}\label{q1}
		&&\left\|\left|u_n\right|^{\alpha}u_n-\left|u\right|^{\alpha}u\right\|_{L^\gamma \left(I,L^\rho\right)}\notag\\
		&\lesssim& \left(\left\|u_n\right\|^\alpha_{L^{\left(\alpha+1\right)\gamma }\left(I,L^{\left(\alpha+1\right)\rho}\right)}+\left\|u\right\|^\alpha_{L^{\left(\alpha+1\right)\gamma }\left(I,L^{\left(\alpha+1\right)\rho}\right)}\right)\left\|u_n-u\right\|_{L^{\left(\alpha+1\right)\gamma }\left(I,L^{\left(\alpha+1\right)\rho}\right)}\notag\\
		&\lesssim &\left(\left\|\left|u_n\right|^{\alpha}u_n\right\|_{L^\gamma \left(I,L^\rho\right)}^{\frac{\alpha}{\alpha+1}}+\left\|\left|u\right|^{\alpha}u\right\|_{L^\gamma \left(I,L^\rho\right)}^{\frac{\alpha}{\alpha+1}}\right)\left\|\left|u_n-u\right|^{\alpha}\left(u_n-u\right)\right\|_{L^\gamma \left(I,L^\rho\right)}^{\frac{1}{\alpha+1}}.
		\end{eqnarray}
		Moreover, applying Lemma \ref{l4} and the embedding  (\ref{a1}), we deduce that
		\begin{eqnarray}\label{q2}
		&&\left\|\left|u_n-u\right|^{\alpha}\left(u_n-u\right)\right\|_{L^\gamma \left(I,L^\rho\right)}\notag\\
		&\lesssim& T^{\alpha \left(\frac{1}{\alpha+2}-\frac{1}{\gamma }\right)}\left(\left\|u_n-u\right\|_{L^\gamma \left(I,H^{4,\rho}\right)}^{\alpha+1}+\left\|\partial_tu_n-\partial_tu\right\|_{L^\gamma \left(I,H^{4,\rho}\right)}^{\alpha+1}\right)\notag\\
		&&+\left\|\phi_n-\phi\right\|_{H^4}^\alpha F\left(\phi_n-\phi,T\right).
		\end{eqnarray}
		Similarly, we have
		\begin{eqnarray}\label{q3}
		\left\|\left|u_n\right|^{\alpha}u_n\right\|_{L^\gamma \left(I,L^\rho\right)}
		&\lesssim& T^{\alpha \left(\frac{1}{\alpha+2}-\frac{1}{\gamma }\right)}\left(\left\|u_n\right\|_{L^\gamma \left(I,H^{4,\rho}\right)}^{\alpha+1}+\left\|\partial_tu_n\right\|_{L^\gamma \left(I,H^{4,\rho}\right)}^{\alpha+1}\right)\notag\\
		&&+\left\|\phi_n\right\|_{H^4}^\alpha F\left(\phi_n,T\right),
		\end{eqnarray}
		and
		\begin{eqnarray}\label{q4}
		\left\|\left|u\right|^{\alpha}u_n\right\|_{L^\gamma \left(I,L^\rho\right)}
		&\lesssim& T^{\alpha \left(\frac{1}{\alpha+2}-\frac{1}{\gamma }\right)}\left(\left\|u\right\|_{L^\gamma \left(I,H^{4,\rho}\right)}^{\alpha+1}+\left\|\partial_tu\right\|_{L^\gamma \left(I,H^{4,\rho}\right)}^{\alpha+1}\right)\notag\\
		&&+\left\|\phi\right\|_{H^4}^\alpha F\left(\phi,T\right).
		\end{eqnarray}
		Estimates  (\ref{q1}), (\ref{q2}), (\ref{q3}) and (\ref{q4}) imply
		\begin{eqnarray}\label{121613}
		&&\left\|\left|u_n\right|^{\alpha}u_n-\left|u\right|^{\alpha}u\right\|_{L^\gamma \left(I,L^\rho\right)}\notag\\
		&\lesssim& T^{\frac{\alpha}{\alpha+1}\left(\frac{1}{\alpha+2}-\frac{1}{\gamma }\right)}\left(M^\alpha+M^{\frac{\alpha \left(2\alpha+1\right)}{\alpha+1}}\right)\left\|u_n-u\right\|_{L^\gamma \left(I,H^{4,\rho}\right)}\notag\\
		&&+C\left(M,T\right)\left(\left\|\partial_tu_n-\partial_tu\right\|_{L^\gamma \left(I,L^\rho\right)}+\left\|\phi_n-\phi\right\|_{H^4}^{\frac{\alpha}{\alpha+1}}\right),
		\end{eqnarray}
		where we also used  (\ref{12111}) and the boundedness of $u_n,u$ in (\ref{M}), (\ref{M1}). Applying (\ref{1118200}) and (\ref{121613}), we obtain
		\begin{eqnarray}\label{121614}
	&&	\left\|u_n-u\right\|_{L^\gamma   \left(I,H^{4,\rho}\right)}\notag\\
		&\le&C_5 T^{\frac{\alpha}{\alpha+1}\left(\frac{1}{\alpha+2}-\frac{1}{\gamma }\right)}\left(M^\alpha+M^{\frac{\alpha \left(2\alpha+1\right)}{\alpha+1}}\right)\left\|u_n-u\right\|_{L^\gamma \left(I,H^{4,\rho}\right)}\notag\\
		&&+C\left(M,T\right)\left(\left\|\partial_tu_n-\partial_tu\right\|_{L^\gamma \left(I,L^\rho\right)}+\left\|\phi_n-\phi\right\|_{H^4}^{\frac{\alpha}{\alpha+1}}\right).
		\end{eqnarray}
	Note that  $C_5 T^{\frac{\alpha}{\alpha+1}\left(\frac{1}{\alpha+2}-\frac{1}{\gamma }\right)}\left(M^\alpha+M^{\frac{\alpha \left(2\alpha+1\right)}{\alpha+1}}\right)\le \frac{1}{2}$ by (\ref{9171}), it follows from Claim \ref{c2} and (\ref{121614}) that  Claim \ref{cc1} holds.
		\end{proof}
	\begin{claim}\label{c3}
		Given any $(q,r)\in\Lambda_b$, we have $\left\|u_n-u\right\|_{L^q \left(I,H^{4,r}\right)}\rightarrow0$ as $n\rightarrow\infty$.
	\end{claim}
	\begin{proof}
		Firstly, we prove that $u_n$ is uniformly bounded in $L^\infty \left(I,H^4\right)$:
		\begin{eqnarray}\label{12244}
		\sup_{n\ge  n_1}\left\|u_n\right\|_{L^\infty \left(I,H^4\right)}<\infty.
		\end{eqnarray}
		Analogously to  (\ref{7.15}) and (\ref{ab}), we have
		\begin{eqnarray}\label{12262}
		&&\|u_n\|_{H^{1,\infty }(I,L^2)}\notag\\&\lesssim& \left\|\phi_n\right\|_{H^4}+\left\|\phi_n\right\|_{H^4}^{\alpha+1}+T^{1-\frac{\alpha+2}{\gamma }}\left\|u_n\right\|_{L^\gamma   \left(I,H^{4,\rho}\right)}^\alpha \left\|u_n\right\|_{H^{1,\gamma   } \left(I,L^\rho\right)},
		\end{eqnarray}
		and
		\begin{eqnarray}\label{12263}
		\|\Delta^2u_n\|_{L^\infty (I,L^2)}\lesssim \|\partial_t u_n\|_{L^\infty (I,L^2)}+\|u_n\|_{L^\infty (I,L^2)}+\||u_n|^\alpha u_n\|_{L^\infty (I,L^2)}.
		\end{eqnarray}
		Since $u_n$ is uniformly bounded in $L^\gamma \left(I,H^{4,\rho}\right)\cap H^{1,\gamma }\left(I,L^\rho\right)$ by (\ref{M1}), it follows from (\ref{12262}), (\ref{12263}) and  Lemma \ref{l3} that (\ref{12244}) holds.
 	
		We now resume the proof of Claim \ref{c3}. Analogously to (\ref{ab}), we have
		\begin{eqnarray}\label{121615}
		\left\|u_n-u\right\|_{L^q  \left(I,H^{4,r}\right)}
		\lesssim\|u_n-u\|_{H^{1,q} \left(I,L^r\right)}+\left\||u_n|^\alpha u_n-|u|^\alpha u\right\|_{L^q  \left(I,L^r\right)}.
		\end{eqnarray}
		Next, we  prove that
		\begin{eqnarray}\label{12116}
		\lim_{n\to\infty}\left\||u_n|^\alpha u_n-|u|^\alpha u\right\|_{L^q \left(I,L^r\right)}=0.
		\end{eqnarray}
		Since $\left\|u_n-u\right\|_{L^\gamma \left(I,H^{4,\rho}\right)}\rightarrow0$ as $n\rightarrow \infty $ by Claim \ref{cc1}, we can apply Lemma \ref{l2} to  deduce that
		\begin{eqnarray}\label{12264}
		\lim_{n\to\infty}\left\||u_n|^\alpha u_n-|u|^\alpha u\right\|_{L^\infty \left(I,L^2\right)\cap L^2\left(I,L^{\fc{2N}{N-4}}\right)}=0,
		\end{eqnarray}
		where we also used  (\ref{M}), (\ref{M1}) and (\ref{12244}).
	Applying 	(\ref{12264}) and  H\"older's inequality, we obtain   (\ref{12116}).
		
		Claim \ref{c3} is now an immediate consequence of (\ref{121615}), (\ref{12116}) and Claim \ref{c2}.
	\end{proof}
	We now resume the proof of the continuous dependence. It follows from Claims  \ref{c2} and \ref{c3} that $u_n\to u$ in $L^q \left([0,T],H^{4,r}\right)\cap H^{1,q}\left([0,T],L^r\right)$ for any biharmonic admissible pair $\left(q,r\right)\in\Lambda_b$. In particular, we have $\left\|u_n\left(T\right)-u\left(T\right)\right\|_{H^4}\underset{n\to \infty }\longrightarrow 0$. Arguing as previously, we deduce that the solution $u_n$ exists on the interval $[T,2T]$ for $n\ge n_2$ and that $u_n\to u$ in $L^q\left([T,2T],H^{4,r}\right)\cap H^{1,q}\left([T,2T],L^r\right)$ for any biharmonic admissible pair $\left(q,r\right)\in\Lambda_b$. Iterating finitely many times like this, we get the continuous dependence on the interval $[0,A]$.
\end{proof}
\section{Proof of Theorem \ref{T1.01}}\label{s5}
In this section, we prove Theorem \ref{T1.01}. For the convenience of the reader, we briefly sketch the proof. Indeed, readers seeking a fuller treatment of certain details can consult Section \ref{s4}. Throughout this section, we denote $I=[0,T]$ with $T>0$.

We first give an analogue of Proposition \ref{P4.1}.
\begin{proposition}\label{P5.1}
	Assume that $\lambda\in\mathbb{C},\mu=\pm1$ or $0, N\ge9,\alpha=\frac{8}{N-8},\phi\in H^4$ and $M$ be sufficiently small such that
	\begin{eqnarray}\label{s5M}
	\left(C_6+C_7\right)M^\alpha\le\frac12,
	\end{eqnarray}
	where $C_6,C_7$ are the constants in (\ref{12124}) and (\ref{12125}), respectively. Let $T>0$ and suppose further that
	\begin{eqnarray}\label{s5T}
	C_6\left(1+\left\|\phi\right\|_{H^4}^\alpha\right)F\left(\phi,T\right)\le \frac{M}{2}.
	\end{eqnarray}
	It follows that   there exists a unique solution $u\in C\left([0,T],H^4\right)\cap L^\gamma \left([0,T],H^{4,\rho}\right)$ to the Cauchy problem (\ref{NLS})-(\ref{12241}) with
	$
	\left\|u\right\|_{H^{1,\gamma  }\left(I,L^\rho\right)\cap L^\gamma  \left(I,H^{4,\rho}\right)}\le M.
	$
	 Moreover, we have  $u,u_{t}, \Delta^2 u \in C\left([0, T], L^{2}\left(\mathbb{R}^{N}\right)\right) \cap L^{q}\left((0, T), L^{r}\left(\mathbb{R}^{N}\right)\right)$ for every biharmonic admissible pair $(q,r)\in\Lambda_b$.
	
\end{proposition}
\begin{proof}
	We look for a fixed point of the map $S$ on the space $X_{T,M}$, where $S,X_{T,M}$ were defined in (\ref{12113}) and (\ref{XTM}) respectively.
	
	We first show that $S$ maps $X_{T,M}$ into itself. Note that $\gamma =\alpha+2$ in the critical case $\alpha=\frac{8}{N-8}$, so that  by  Lemma \ref{l1} we have
	\begin{eqnarray}\label{12121}
	\|Su\|_{H^{1,\gamma}(I,L^\rho)}\lesssim F\left(\phi,T\right)+\left\|u\right\|_{L^\gamma \left(I,H^{4,\rho}\right)}^\alpha \left\|u\right\|_{H^{1,\gamma } \left(I,L^\rho\right)}.
	\end{eqnarray}
	Our next step is to estimate  $\|Su\|_{L^\gamma(I,H^{4,\rho})}$. Similarly to (\ref{ab}), we have
	\begin{eqnarray}\label{12122}
	\left\|\Delta^2\left(Su\right)\right\|_{L^\gamma \left(I,L^\rho\right)}\lesssim \left\|Su\right\|_{H^{1,\gamma }\left(I,L^\rho\right)}+\left\|\left|u\right|^{\alpha}u\right\|_{L^\gamma \left(I,L^\rho\right)}.
	\end{eqnarray}
	Moreover, it follows from Lemma \ref{l4} and the embedding (\ref{a1}) that
	\begin{eqnarray}\label{12123}
	\left\|\left|u\right|^{\alpha}u\right\|_{L^\gamma \left(I,L^\rho\right)}\lesssim \left(\left\|u\right\|_{L^\gamma \left(I,H^{4,\rho}\right)}^{\alpha+1}+\left\|\partial_tu\right\|_{L^\gamma \left(I,L^\rho\right)}^{\alpha+1}\right)+\left\|\phi\right\|_{H^4}^\alpha F\left(\phi,T\right).
	\end{eqnarray}
	It now follows from  (\ref{12121}), (\ref{12122}) and (\ref{12123}) that, for any $u\in X_{T,M}$,
	\begin{eqnarray}\label{12124}
	\left\|Su\right\|_{H^{1,\gamma }\left(I,L^\rho\right)\cap L^\gamma \left(I,H^{4,\rho}\right)}\le C_6 F\left(\phi,T\right)+C_6M^{\alpha+1}+C_6\left\|\phi\right\|_{H^4}^\alpha F\left(\phi,T\right).
	\end{eqnarray}
	Applying (\ref{s5M}), (\ref{s5T}) and (\ref{12124}), we conclude that
	\begin{eqnarray*}
		\left\|Su\right\|_{H^{1,\gamma }\left(I,L^\rho\right)\cap L^\gamma \left(I,H^{4,\rho}\right)}\le M.
	\end{eqnarray*}
	On the other hand,  given  $u,v\in X_{T,M}$, we can apply the same method as that used to derive (\ref{12110}) to obtain
	\begin{eqnarray}\label{12125}
	d(Su,Sv)
	&\lesssim& \left(\|u\|_{L^\gamma(I,H^{4,\rho})}^\alpha+\|v\|_{L^\gamma(I,H^{4,\rho})}^\alpha\right)\|u-v\|_{L^{\gamma}(I,L^{\rho})}\notag\\
	&\le& C_7M^\alpha d(u,v)\le\frac12 d\left(u,v\right).
	\end{eqnarray}
	So we prove that $S$ is a contraction on the space $X_{T,M}$. Using Banach's fixed-point argument, we obtain   a unique solution $u\in C\left([0,T],H^4\right)\cap L^\gamma \left([0,T],H^{4,\rho}\right)$ to the Cauchy problem (\ref{NLS})-(\ref{12241}). Moreover, arguing as in the last part  of Proposition \ref{P4.1}, we obtain the further regularity properties. This finishes the proof of Proposition \ref{P5.1}.
\end{proof}
\begin{proof}[\textbf{Proof of Theorem \ref{T1.01}}]
	We now resume the proof of Theorem \ref{T1.01}.  We consider only the positive time direction. A corresponding conclusion for  reverse direction follows similarly. We    proceed  as  in  the proof  of  Theorem \ref{T1.00}: Using Proposition \ref{P5.1} and the uniqueness in  Appendix, we obtain  a unique maximal solution $u\in C\left([0,T_{\text{max}}),H^4\right)$ to the Cauchy problem (\ref{NLS})-(\ref{12241}) with  $u,u_{t}, \Delta^2 u \in$ $C\left([0, T_{\text{max}}), L^{2}\left(\mathbb{R}^{N}\right)\right) \cap L^{q}_{\text{loc}}\left((0, T_{\text{max}}), L^{r}\left(\mathbb{R}^{N}\right)\right)$ for every biharmonic admissible pair $(q,r)\in\Lambda_b$.
	
	It remains to  prove the blowup alternative (\ref{bl3}). Suppose by contraction that $T_{\text{max}}<\infty$ and
	\begin{eqnarray}\label{12177}
	\|u\|_{L^\gamma((0,T_{\text{max}}),L^{\frac{\rho\alpha}{\rho-2}})}<\infty.
	\end{eqnarray}
	In fact, in the $H^4$ critical case $\alpha=\frac{8}{N-8}$, we have $\gamma =\frac{2N-8}{N-8}$ and $\frac{\rho\alpha}{\rho-2}=\frac{2N\left(N-4\right)}{\left(N-8\right)^2}$. We reach a contradiction by showing that the solution $u$ can be extended beyond the maximal interval $(0,T_{\text{max}})$. Firstly, we claim that
	\begin{equation}\label{7216}
	\left\|u\right\|_{H^{1,\gamma }\left([0,T_{\text{max}}\right),L^\rho)\cap H^{1,\infty  }\left([0,T_{\text{max}}\right),L^2)}<\infty,
	\end{equation}
	and
	\begin{eqnarray}\label{12176}
	\left\|\Delta^2u\right\|_{L^{\gamma }\left([0,T_{\text{max}}\right),L^\rho)}<\infty.
	\end{eqnarray}
	Indeed, it follows from  monotone convergence theorem that there exists  a $T_0\in[0,T_{\text{max}})$ such that
	\begin{equation}
	C_{8}\|u\|_{L^\gamma([T_0,T_{\text{max}}),L^{\frac{\rho\alpha}{\rho-2}})}^\alpha\leq\fc12,\notag
	\end{equation}
	where  $C_{8}$ is the constant in (\ref{7215}). Changing $u\left(\cdot\right)$ to $u\left(T_0+\cdot\right)$ and $\phi$ to $u\left(T_0\right)$, we can assume that $T_0=0$, so that
	\begin{equation}\label{7214}
	C_{8}\|u\|_{L^\gamma([0,T_{\text{max}}),L^{\frac{\rho\alpha}{\rho-2}})}^\alpha\leq\fc12.
	\end{equation}
	On the other hand, it follows from  (\ref{12113}), (\ref{NLS2}), (\ref{12111}) and  Strichartz' estimate that
	\begin{eqnarray}\label{7215}
	&&\|u\|_{H^{1,\gamma }([0,T),L^\rho)\cap H^{1,\infty  }\left([0,T\right),L^2)}\notag\\
	&\leq& C_{8}\left(\|\phi\|_{H^4}+\|\phi\|_{H^4}^{\alpha+1}\right)
	+C_8\left\|u\right\|_{L^\gamma([0,T),L^{\frac{\rho\alpha}{\rho-2}})}^\alpha\|u\|_{H^{1,\gamma }([0,T),L^\rho)}
	\end{eqnarray}
	for all $0<T<{T_{\text{max}}}$.
	This inequality together with (\ref{7214}) yields (\ref{7216}).
	
	Our next goal is to prove (\ref{12176}). Applying  the same method as that used to derive (\ref{ab}), we obtain
	\begin{eqnarray}\label{12178}
	\left\|\Delta^2u\right\|_{L^\gamma \left([0,T_{\text{max}}),L^\rho\right)} \lesssim \left\|u\right\|_{H^{1,\gamma } \left([0,T_{\text{max}}),L^\rho\right)}+\left\|\left|u\right|^{\alpha}u\right\|_{L^\gamma \left([0,T_{\text{max}}),L^\rho\right)}.
	\end{eqnarray}
	Note that  $q_0=p_0=\frac{\rho\alpha}{\rho-2}$ in the critical case $\alpha=\frac{8}{N-8}$ by (\ref{12171}),  (\ref{12201}),  (\ref{12172}), (\ref{12202}) and the definition of $\rho$ in (\ref{7181}), we can apply (\ref{12178}), Lemma \ref{l4}, (\ref{12177}) and (\ref{7216}) to obtain  (\ref{12176}).

	We  now apply (\ref{12177}) and (\ref{7216}) to derive a contradiction. Using the same method as that used to derive (\ref{12244}), we deduce  that  $u\in L^\infty ([0,$ $T_{\text{max}}),H^4)$. Moreover, applying  Strichartz' estimate (\ref{sz}) and  Sobolev's embedding $H^4 \hookrightarrow L^{2\left(\alpha+1\right)}$, we conclude that
	\begin{eqnarray}\label{7217}
	&&\|e^{i(t-T_\ep)(\Delta^2+\mu\Delta)}(\Delta^2+\mu\Delta)u(T_\ep)\|_{L^\gamma([T_\ep,T_{\text{max}}),L^\rho)}\notag\\
	&&\quad+\|e^{i(t-T_\ep)(\Delta^2+\mu\Delta)}|u|^\alpha u(T_\ep)\|_{L^\gamma([T_\ep,T_{\text{max}}),L^\rho)}\notag\\
	&&\qquad+\|e^{i(t-T_\ep)(\Delta^2+\mu\Delta)} u(T_\ep)\|_{L^\gamma([T_\ep,T_{\text{max}}),L^\rho)}\notag\\
	&\lesssim& \left\|u\right\|_{L^\infty \left([0,T_{\text{max}}),H^4\right)}+\left\|u\right\|_{L^\infty \left([0,T_{\text{max}}),H^4\right)}^{\alpha+1}<\infty,
	\end{eqnarray}
	for any $0<T_\ep<T_{\text{max}}$.
	It now follows from the  monotone convergence theorem that we can choose $T_\ep$ approaches to $T_{\text{max}}$ such that
	\begin{eqnarray}
	&&C_6\left(1+\left\|u\left(T_\ep\right)\right\|_{H^4}^\alpha\right)\left(\|e^{i(t-T_\ep)(\Delta^2+\mu\Delta)}(\Delta^2+\mu\Delta)u(T_\ep)\|_{L^\gamma([T_\ep,T_{\text{max}}),L^\rho)}\right.\notag\\
	&&\quad+\|e^{i(t-T_\ep)(\Delta^2+\mu\Delta)}|u|^\alpha u(T_\ep)\|_{L^\gamma([T_\ep,T_{\text{max}}),L^\rho)}\notag\\
	&&\qquad+\left.\|e^{i(t-T_\ep)(\Delta^2+\mu\Delta)} u(T_\ep)\|_{L^\gamma([T_\ep,T_{\text{max}}),L^\rho)}\right)\notag\\
	&\le&\frac{M}{4},
	\end{eqnarray}
	where the constants $M,C_6$ are the constants in (\ref{s5M}) and (\ref{12124}) respectively. Finally, applying  the monotone convergence theorem again,  we deduce that there exists $\delta_0>0$ such that
	\begin{eqnarray}\label{12265}
	&&C_6\left(1+\left\|u\left(T_\ep\right)\right\|_{H^4}^\alpha\right)\left(\|e^{i(t-T_\ep)(\Delta^2+\mu\Delta)}(\Delta^2+\mu\Delta)u(T_\ep)\|_{L^\gamma([T_\ep,T_{\text{max}}+\delta_0),L^\rho)}\right.\notag\\
	&&\quad+\|e^{i(t-T_\ep)(\Delta^2+\mu\Delta)}|u|^\alpha u(T_\ep)\|_{L^\gamma([T_\ep,T_{\text{max}}+\delta_0),L^\rho)}\notag\\
	&&\qquad+\left.\|e^{i(t-T_\ep)(\Delta^2+\mu\Delta)} u(T_\ep)\|_{L^\gamma([T_\ep,T_{\text{max}}+\delta_0),L^\rho)}\right)\notag\\
	&\le&\frac{M}{2}.
	\end{eqnarray}
	It now follows from  Proposition \ref{P5.1} and (\ref{12265}) that there exists a unique solution $v\in C([T_\ep,$ $T_{\text{max}}+\delta_0),H^4)$ to the equation (\ref{NLS}) with initial datum $u(T_\ep)$ at time $t=T_\ep$.  Moreover, by the uniqueness established in  Appendix, we see that $u=v$ on $[T_\ep,T_{\text{max}}]$, which implies that the solution can be extended beyond the maximal interval $\left(0,T_{\text{max}}\right)$. This is a contradiction, thereby completing the proof of the blowup alternative (\ref{bl3}).

	In the rest of this section, we prove the continuous dependence of the solution  map.

	Let $M>0$ satisfy
	\begin{eqnarray}\label{12151}
	\left(C_6+C_7+C_9+C_{10}\right)M^\alpha+C_{11} \left(M^\alpha+M^{\frac{\alpha \left(2\alpha+1\right)}{\alpha+1}}\right)\le\frac12
	\end{eqnarray}
	where  $C_6,C_7,C_9,C_{10},C_{11}$ are the constants  in (\ref{12124}),  (\ref{12125}), (\ref{12128}), (\ref{12127}) and (\ref{121617}) respectively. Moreover, we fix $T>0$ sufficiently small such that
	\begin{equation}\label{s5.T}
	C_6\left(1+\left\|\phi\right\|_{H^4}^\alpha\right)F\left(\phi,T\right)\le\frac M4.
	\end{equation}

	Since  $\phi_n\rightarrow\phi$ in $H^4$, we can find  a positive $n_1$ such that $C_6\left(1+\left\|\phi_n\right\|_{H^4}^\alpha\right)F\left(\phi_n,T\right)$ $\le \frac M2$
	for every $n\ge n_1$. Then it follows  from Proposition \ref{P5.1} that for every $n\ge n_1$, the equation (\ref{1221})
	admits a unique solution $u_n\in C(I,H^4)\cap  L^\gamma(I,B^{4}_{\rho,2})$ and $u_n$ is uniformly bounded:
	\begin{equation}
	\|u_n\|_{L^\gamma  \left(I,H^{4,\rho}\right)\cap H^{1,\gamma }\left(I,L^\rho\right)} \leq M,\ \text{for } \forall\  n\ge n_1,\notag
	\end{equation}where  $M$ is the  constant in (\ref{12151}).
	Furthermore, we also have the boundedness for $u$
	\begin{eqnarray}
	\|u\|_{ L^\gamma  \left(I,H^{4,\rho}\right)\cap H^{1,\gamma }\left(I,L^\rho\right)} \leq M.\notag
	\end{eqnarray}
	The proof of the continuous dependence will proceed by the following claims.
	\begin{claim}\label{c51}
		Given any $(q,r)\in\Lambda_b$, we have $\left\|u_n-u\right\|_{H^{1,q} \left(I,L^r\right)}\rightarrow0$ as $n\rightarrow\infty$.
	\end{claim}
	\begin{proof}
		We first show that  $\left\|u_n-u\right\|_{L^q L^r}\rightarrow0$ as $n\rightarrow\infty$. Analogously to  (\ref{12112}), we obtain
		\begin{eqnarray}\label{12128}
		&&\left\|u_n-u\right\|_{L^q \left(I,L^r\right)}\notag\\
		&\lesssim&\left\|\phi_n-\phi\right\|_{L^2}+\left(\|u_n\|_{L^\gamma(I,H^{4,\rho})}^\alpha+\|u\|_{L^\gamma(I,H^{4,\rho})}^\alpha\right)\|u_n-u\|_{L^{\gamma}(I,L^{\rho})}\notag\\
		&\le&C_9\left\|\phi_n-\phi\right\|_{L^2}+C_9M^\alpha\left\|u_n-u\right\|_{L^q \left(I,L^r\right)}.
		\end{eqnarray}
		Since $C_9M^\alpha\le \frac{1}{2}$ by (\ref{12151}), it follows from  (\ref{12128}) that $\left\|u_n-u\right\|_{L^q \left(I,L^r\right)}\lesssim \left\|\phi_n-\phi\right\|_{L^2}$ $\rightarrow0$ as $n\rightarrow \infty $.
		
		Next, it follows from  Lemma \ref{l1} and the boundedness of $u_n,u$, we have, for any $R>0$
		\begin{eqnarray*}
			&&\limsup_{n\rightarrow\infty}\|u_n-u\|_{H^{1,q}\left(I,L^r\right)\cap H^{1,\gamma }\left(I,L^\rho\right)}\notag\\
			&\lesssim&M^\alpha \limsup_{n\rightarrow\infty}\|u_n-u\|_{H^{1,\gamma } \left(I,L^\rho\right)}+M^\alpha\|\left(\partial_tu\right)^R\|_{L^\gamma \left(I,L^\rho\right)}.
		\end{eqnarray*}
		Since  $R>0$ is arbitrary, we can let $R\rightarrow\infty$ to obtain
		\begin{eqnarray}\label{12127}
		\limsup_{n\rightarrow\infty}\|u_n-u\|_{H^{1,q}\left(I,L^r\right)\cap H^{1,\gamma }\left(I,L^\rho\right)}\le C_{10}M^\alpha \limsup_{n\rightarrow\infty}\|u_n-u\|_{H^{1,\gamma } \left(I,L^\rho\right)}
		\end{eqnarray}
		Since  $C_{10}M^\alpha\le\fc12$ by (\ref{12151}), it  follows from (\ref{12127}) that  $\limsup_{n\rightarrow\infty}\|u_n-u\|_{H^{1,q}\left(I,L^r\right)}$ $=0$. This finishes the proof of Claim \ref{c51}.
	\end{proof}
	\begin{claim}\label{cc2}
		$	\left\|u_n-u\right\|_{L^\gamma   \left(I,H^{4,\rho}\right)}\rightarrow0$ as $n\rightarrow\infty$.
	\end{claim}
	\begin{proof}
		Using the same method as that used to derive (\ref{121614}), we obtain
		\begin{eqnarray}\label{121617}
	&&\left\|u_n-u\right\|_{L^\gamma   \left(I,H^{4,\rho}\right)}\notag\\
	&\le&C_{11} \left(M^\alpha+M^{\frac{\alpha \left(2\alpha+1\right)}{\alpha+1}}\right)\left\|u_n-u\right\|_{L^\gamma \left(I,H^{4,\rho}\right)}\notag\\
	&&+C\left(M,T\right)\left(\left\|\partial_tu_n-\partial_tu\right\|_{L^\gamma \left(I,L^\rho\right)}+\left\|\phi_n-\phi\right\|_{H^4}^{\frac{\alpha}{\alpha+1}}\right).
	\end{eqnarray}
	 Since $C_{11} \left(M^\alpha+M^{\frac{\alpha \left(2\alpha+1\right)}{\alpha+1}}\right)\le \frac{1}{2}$ by (\ref{s5.T}), it  follows from Claim \ref{c51} and (\ref{121617})  that Claim \ref{cc2} holds.
	\end{proof}
	\begin{claim}\label{c52}
		Given any $(q,r)\in\Lambda_b$, we have $\left\|u_n-u\right\|_{L^q \left(I,H^{4,r}\right)}\rightarrow0$ as $n\rightarrow\infty$.
	\end{claim}
	\begin{proof}
		Using  the same method as that used to derive Claim \ref{c3}, we obtain Claim \ref{c52}.
	\end{proof}
	It follows from Claims \ref{c51} and \ref{c52} that the continuous dependence of the solution map holds on the interval $[0,T]$. Finally, by a standard iteration argument, we obtain the continuous dependence on $[0,A]$. 
\end{proof}

\section{THE BLOW-UP ANSATZ}\label{section 2}
The rest of this paper is devoted to the proof of Theorem \ref{T1.1}. In this section,  we construct inductively an appropriate  blow-up ansatz.
The first candidate $U_0$ is defined by (\ref{2.03}) below. $U_0$ is a natural candidate, since it is an explicit blowing-up solution of the ODE  $i \partial_tU_0+\lambda|U_0|^\alpha U_0=0$. Moreover, the error term $\Delta^2U_0+\mu\Delta U_0$ is lower order than both $i\partial_tU_0$ and $\lambda \left|U_0\right|^{\alpha}U_0$. (See Lemma \ref{L2.1} below.) Since  $\Delta^2U_0+\mu\Delta U_0$  is of order $\left(-t\right)^{-\frac{4}{k}}\left|U_0\right|\lesssim \left(-t\right)^{-\frac{1}{\alpha}-\frac{4}{k}}$, the error term is not integrable in time near the singularity when $\alpha$ is small.  To treat any subcritical or critical $\alpha$, we refine the approximate solution in the spirit of  Cazenave-Han-Martel \cite{Ca2} to reduce the singularity of the error term at any order of $\left(-t\right)$. See (\ref{2.013})-(\ref{2.017}) for more details.
Throughout this section,  we assume
\begin{equation}\label{J}
J=\left[\frac{2}{\alpha}+4 \sigma\right]+1
\end{equation}
and \begin{equation}\label{k}
k=\max\{4J+6,  [N \alpha]+1\}
\end{equation}
with
\begin{eqnarray}\label{3.4}
\sigma=\max\{\frac4\delta, \frac4{\alpha(1-\delta)}, \frac{2^{\alpha+2}|\lambda|M(-\alpha\text{Im}\lambda)^{-1}}{\min\{\alpha,1\}\left(1-\delta\right)}, g(\alpha)\},
\end{eqnarray}
where
\begin{eqnarray}\label{3.1}
\delta=\min\{\frac1{10}, \frac{\alpha}{\alpha+4}, f(\alpha)\},
\end{eqnarray}
 $M$ is the constant in Lemma \ref{nonlinear estimate}, $f(\alpha)$ and $g(\alpha)$ are defined by
\begin{equation*}
f(\alpha)=\begin{cases}
1,   & \text{if } 0<\alpha\leq 1, \\
\frac{\alpha-1}{2(\alpha+2)},   & \text{if } \alpha>1,
\end{cases}\ \ \
g(\alpha)=\begin{cases}
0,   & \text{if } 0<\alpha\leq 1, \\
\frac2{(\alpha-1)-(\alpha+2)\delta},   & \text{if } \alpha>1.
\end{cases}
\end{equation*}
Let $K$ be any nonempty compact set of $\mathbb{R}^N $  included in the ball of center 0 and radius $R>0$. It is well-known that there exists a smooth function $Z :\mathbb{R}^N\rightarrow [0, \infty)$ which vanishes exactly on $K$ (see  Lemma 1.4 in  \cite{Mo}).   Define the function $A: \mathbb{R}^N\rightarrow[0, \infty)$ by
\begin{equation}
A(x)=(Z(x) \chi(|x|)+(1-\chi(|x|))|x|)^{k}\notag
\end{equation}
where
\begin{equation*}
{\chi \in C^{\infty}(\mathbb{R},  \mathbb{R})},  \ \ {\chi(s)=\left\{\begin{array}{ll}{1},  & {0 \leq s \leq R},  \\ {0},  & {s \geq 2R}, \end{array}\right.} \ \ {\chi^{\prime}(s) \leq 0 \leq \chi(s) \leq 1,  \quad s \geq 0}.
\end{equation*}
It follows that the function $A\in C^{k-1}\left(\mathbb{R}^{N},   \mathbb{R}\right)$,   vanishes exactly on $K$,   satisfies
\begin{equation}\label{2.02}
\left\{\begin{array}{ll}{A \geq 0 \text { and }\left|\partial_{x}^{\beta} A\right| \lesssim A^{1-\frac{|\beta|}{k}}},  & {\text { on } \mathbb{R}^{N} \text { for }|\beta| \leq k-1},  \\ {A(x)=|x|^{k}},  & {\text { for } x \in \mathbb{R}^{N},  |x| \geq 2R}.\end{array}\right.
\end{equation}
Set
\begin{equation}\label{2.03}
U_{0}(t,   x)=(-\text{Im }\lambda)^{-\frac1\alpha}(-\alpha t+A(x))^{-\frac{1}{\alpha}+i \frac{\text{Re}\lambda}{\alpha\text{Im}\lambda}},  \quad t<0,   x \in \mathbb{R}^{N}.
\end{equation}
From   (\ref{71512}),  (\ref{k}), (\ref{2.02}) and (\ref{2.03}),  we have
\begin{equation*}
U_{0} \text { is } C^{\infty} \text { in } t<0 \text { and } C^{k-1} \text { in } x \in \mathbb{R}^{N},
\end{equation*}
\begin{equation}\label{2.05}
i\partial_{t} U_{0}+\lambda |U_0|^\alpha U_0=0,   \quad t<0,   x \in \mathbb{R}^{N},
\end{equation}
\begin{equation}\label{2.06}
|U_{0}|=(-\text{Im}\lambda)^{-\frac1\alpha}(-\alpha t+A(x))^{-\frac{1}{\alpha}} \leq(-\alpha\text{Im}\lambda)^{-\frac{1}{\alpha}}(-t)^{-\frac{1}{\alpha}},
\end{equation}
and
\begin{equation}\notag
\partial_{t}\left|U_{0}\right|=-\text{Im}\lambda|U_0|^{\alpha+1}\geq0.
\end{equation}
Next, we collect the estimates on $U_0$ which are from \cite{Ca2,xuan}.
\begin{lemma}\label{L2.1}
	Assume (\ref{71512}),  (\ref{k}), (\ref{2.02}), and let $U_0$ be given by (\ref{2.03}). If  $p\ge1$, then
	\begin{equation}\label{2.9}
	\left\|U_{0}(t)\right\|_{L^{p}} \lesssim(-t)^{-\frac{1}{\alpha}}
	\end{equation}
	for  $-1 \leq t<0$. In addition,   for every $\rho \in \mathbb{R},   \ell \in \mathbb{N}$ and $|\beta| \leq k-1$,
	\begin{gather}
	\left|\partial_{x}^{\beta}\left(\left|U_{0}\right|^{\rho}\right)\right|  \lesssim\left|U_{0}\right|^{\rho+\frac{\alpha}{k}|\beta|} \lesssim(-t)^{-\frac{|\beta|}{k}}\left|U_{0}\right|^{\rho}, \notag
	\\\left|\partial_{x}^{\beta}\left(\left|U_{0}\right|^{\rho-1} U_{0}\right)\right|  \lesssim\left|U_{0}\right|^{\rho+\frac{\alpha}{k}|\beta|} \lesssim(-t)^{-\frac{|\beta|}{k}}\left|U_{0}\right|^{\rho}, \label{12261}
	\\ \left|\partial_t\partial_{x}^{\beta}\left(\left|U_{0}\right|^{\rho-1} U_{0}\right)\right|  \lesssim\left|U_{0}\right|^{-1+\rho+\frac{\alpha}{k}|\beta|} \lesssim(-t)^{-1-\frac{|\beta|}{k}}\left|U_{0}\right|^{\rho}, \notag\end{gather}
	for all $x \in \mathbb{R}^{N},   t<0,  $ and
	\begin{equation}\notag
	U_{0} \in C^{\infty}\left((-\infty,   0),   H^{k-1}\left(\mathbb{R}^{N}\right)\right).
	\end{equation}
	Furthermore,   for any $x_{0} \in \mathbb{R}^{N}$ such that $A\left(x_{0}\right)=0,  $ for any $r>0,  -1 \leq t<0$ and
	$1 \leq p \leq \infty$,
	\begin{equation}\label{2.012}
	(-t)^{-\frac{1}{\alpha}+\frac{N}{p k}}\lesssim  \left\|U_{0}(t)\right\|_{L^{p}\left(\left|x-x_{0}\right|<r\right)}.
	\end{equation}
\end{lemma}

Next,  we refine the approximate solution in the spirit of  Cazenave-Han-Martel \cite{Ca2}.
More precisely, we consider the linearization of the equation (\ref{2.05}),
\begin{equation}\label{2.013}
i\partial_{t} w+\lambda \frac{\alpha+2}{2}\left|U_{0}\right|^{\alpha} w+\lambda \frac{\alpha}{2}\left|U_{0}\right|^{\alpha-2} U_{0}^{2} \overline{w}=0.
\end{equation}
Equation (\ref{2.013}) has two particular solutions $w=i U_{0}$ and $w=$
$\partial_{t} U_{0}=i\lambda\left|U_{0}\right|^{\alpha} U_{0}$.  By means of variation of parameters,  it is not hard to see that the corresponding nonhomogeneous equation
\begin{equation}\notag
i\partial_{t} w+\lambda(\frac{\alpha+2}2|U_0|^\alpha w+\frac\alpha2|U_0|^{\alpha-2}U_0^2\overline{w})+G=0
\end{equation}
has the solution $w=\mathcal{P}(G),  $ where
\begin{eqnarray*}
	 \mathcal{P}(G)&=&\frac{i\lambda}{\text{Im}\lambda}\left|U_{0}\right|^{\alpha} U_{0} \int_{0}^{t}\left[\left|U_{0}\right|^{-\alpha-2} \text{Im}(\overline{U_{0}} G)\right](s) d s\notag
	\\&&- \frac{i}{\text{Im}\lambda}U_{0} \int_{0}^{t}\left[\left|U_{0}\right|^{-2} \text{Re}(\overline{\lambda U_{0}} G)\right](s) d s.
\end{eqnarray*}
We define $U_{j},   w_{j},   \mathcal{E}_{j}$ by
\begin{equation}\notag
w_{0}=i U_{0},   \quad \mathcal{E}_{0}=i\partial_{t} U_{0}+\Delta^2U_0+\mu\Delta U_{0}+\lambda |U_{0}|^\alpha U_0=\Delta^2U_0+\mu\Delta U_{0}
\end{equation}
and then recursively
\begin{equation}\label{2.017}
w_{j}=\mathcal{P}\left(\mathcal{E}_{j-1}\right),   \quad U_{j}=U_{j-1}+w_{j}, \quad \mathcal{E}_{j}=i\partial_{t} U_{j}+\Delta^2U_j+ \mu\Delta U_{j}+\lambda |U_{j}|^\alpha U_j
\end{equation}
for $j \geq 1,  $ as long as they make sense.  We then have the following estimates.
\begin{lemma}\label{L2.2J}
	Assume (\ref{71512}),  (\ref{J}), (\ref{k}), (\ref{2.02}), and let $U_0,U_j,w_j,\mathcal{E}_j$ be given by (\ref{2.03}) and (\ref{2.017}). There exists $-1 < T<0$ such that the following estimates hold
	\begin{equation}
	\left|\partial_{x}^{\beta}\left(U_{J}-U_{0}\right)\right| \lesssim(-t)^{1-\frac{|\beta|+4}{k}}\left|U_{0}\right|,   \quad \ 0 \leq|\beta| \leq k-1-4 J, \label{2.27J}\\
	\end{equation}
	\begin{equation}\label{2.29J}
	\left|\partial_{x}^{\beta} \mathcal{E}_{J}\right| \lesssim(-t)^{J\left(1-\frac{4}{k}\right)-\frac{|\beta|+4}{k}}\left|U_{0}\right|,   \quad \ 0 \leq|\beta| \leq k-5-4 J,
	\end{equation}
	\begin{gather}
	|\partial_t\mathcal{E}_J|\lesssim(-t)^{-1+J(1-\frac4k)-\frac4k}|U_0|,\label{2.01J}
	\end{gather}
	\begin{equation}\label{2.021J}
	\frac{1}{2}\left|U_{0}\right| \leq\left|U_{J}\right| \leq 2\left|U_{0}\right|,   \quad
	\end{equation}
	\begin{equation}\label{2.31J}
	U_{J} \in C^{1}\left((T,   0),   H^{k-1-4 J}\left(\mathbb{R}^{N}\right)\right),
	\end{equation}
	\begin{equation}
	|\partial_tU_J|\lesssim(-t)^{-1}|U_0|\label{2.02J}
	\end{equation}
	where $ T \leq t<0$,   $x \in \mathbb{R}^{N}$,  and
	\begin{equation}\label{epJ}
	\mathcal{E}_{J}=i\partial_{t} U_{J}+\Delta^2U_J+\mu \Delta U_{J}+\lambda |U_{J}|^\alpha U_J.
	\end{equation}
\end{lemma}
\begin{proof}
	We claim that  for all $0 \leq j \leq \frac{k-6}{4}$ the following inequalities hold \\
	$(1)$ If $0 \leq|\beta| \leq k-1-4 j,  $ then
	\begin{gather}
	\left|\partial_{x}^{\beta} w_{j}\right| \lesssim(-t)^{j\left(1-\frac{4}{k}\right)-\frac{|\beta|}{k}}\left|U_{0}\right|,   \quad T \leq t<0,   x \in \mathbb{R}^{N}, \notag\\
	\left|\partial_t\partial_{x}^{\beta} w_{j}\right| \lesssim(-t)^{-1+j\left(1-\frac{4}{k}\right)-\frac{|\beta|}{k}}\left|U_{0}\right|,   \quad T \leq t<0,   x \in \mathbb{R}^{N}, \notag\\
	\left|\partial_{x}^{\beta}\left(U_{j}-U_{0}\right)\right| \lesssim(-t)^{1-\frac{|\beta|+4}{k}}\left|U_{0}\right|,   \quad T \leq t<0,   x \in \mathbb{R}^{N}, \notag
	\end{gather}
	$(2)$ If $0 \leq|\beta| \leq k-5-4 j,  $ then
	\begin{gather}
	\left|\partial_{x}^{\beta} \mathcal{E}_{j}\right| \lesssim(-t)^{j\left(1-\frac{4}{k}\right)-\frac{|\beta|+4}{k}}\left|U_{0}\right|,   \quad T \leq t<0,   x \in \mathbb{R}^{N},\notag
	\\
	|\partial_t \mathcal{E}_j|\lesssim(-t)^{-1+j(1-\frac4k)-\frac4k}|U_0|\notag,\\
	\notag
	\frac{1}{2}\left|U_{0}\right| \leq\left|U_{j}\right| \leq 2\left|U_{0}\right|,   \quad T \leq t<0,   x \in \mathbb{R}^{N},
	\\\notag
	U_{j} \in C^{1}\left((T,   0),   H^{k-1-4j}\left(\mathbb{R}^{N}\right)\right),
	\\\notag
	|\partial_tU_j|\lesssim(-t)^{-1}|U_0|.
	\end{gather}
	In fact, the proof is an obvious adaptation of  \cite{Ca2,xuan}.   More precisely, we just need to replace all $\frac2k$ with $\frac4k$ throughout the proof of Lemma 3.2 in \cite{Ca2} and Lemma 2.2 in \cite{xuan} by considering the  presence of $\Delta^2 U_j$. Finally, note that $0\le J\le\fc{k-6}4$ by (\ref{k}), we complete the proof of Lemma \ref{L2.2J} by setting $j=J$.
\end{proof}
\begin{lemma}\label{nonlinear estimate}
	Assume  $\alpha>0$, $p\in\mathbb{R}, l\in\mathbb{R}$ and $p+l>0$. There exists a constant $M\geq1$ such that for all complex number $u,v\in \mathbb{C}$,
	\begin{eqnarray}\label{121710}
	\left||u+v|^p(u+v)^l-|u|^pu^l\right|\leq M\left(\left|v\right|^{p+l}+1_{p+l>1}\left|u\right|^{p+l-1}\left|v\right|\right)
	\end{eqnarray}
	and
	\begin{equation}
	\left||u+v|^p(u+v)^l-|u|^pu^l\right|\leq M\left(\left|u\right|^{p+l-1}\left|v\right|+1_{p+l>1}\left|v\right|^{p+l}\right)\label{2.0032}
	\end{equation}
	Moreover, we can decompose
	\begin{equation}
	\partial_t\left(|u+v|^\alpha(u+v)-|u|^\alpha u\right)=I_1\left(u,v\right)+I_2\left(u,v\right),\label{3.0021}
	\end{equation}
	with
	\begin{equation}\label{12163}
	|I_1\left(u,v\right)|\leq M|u|^\alpha|\partial_t v|
	\end{equation}
	and
	\begin{eqnarray}\label{12164}
	|I_2\left(u,v\right)|&\leq& M\left(\left|u\right|^{\alpha-1}\left|v\right|+1_{\alpha>1}\left|v\right|^{\alpha}\right)|\partial_tu|\notag\\
	&&+\left(\left|v\right|^{\alpha}+1_{\alpha>1}\left|u\right|^{\alpha-1}\left|v\right|\right)|\partial_tv|
	\end{eqnarray}
	for all complex number $u,v\in \mathbb{C}$. \end{lemma}
\begin{proof}
	The proof of the estimate (\ref{121710}) is standard, we omit its proof by simplicity.
	
	We now prove  (\ref{2.0032}). It suffices to show that for any $z\in\mathbb{C}$, we have
	\begin{eqnarray}\label{11191}
	\left|\left|1+z\right|^{p}\left(1+z\right)^l-1)\right|\le M\left(\left|z\right|+1_{p+l>1}\left|z\right|^{p+l}\right).
	\end{eqnarray}
	Let $z\in\mathbb{C}$,  $\left|z\right|\ge\frac12$. Note that  $\left|z\right|^{p+l}\lesssim \left|z\right|$ if  $p+l\le1$, so that
	\begin{eqnarray}
	\left|\left|1+z\right|^{p}\left(1+z\right)^l-1)\right|\lesssim \left|z\right|^{p+l}+1\lesssim \left|z\right|+1_{p+l>1}\left|z\right|^{p+l}.\notag
	\end{eqnarray}
	For $\left|z\right|\le\frac12$, we  write
	\begin{eqnarray}\label{12131}
	&&\left|1+z\right|^{p}\left(1+z\right)^l-1=\int_{0}^{1}\frac{d}{d\theta}\left[\left|1+\theta z\right|^{p}\left(1+\theta z\right)^l\right] d\theta.
	\end{eqnarray}
	Note that
	\begin{eqnarray}
	\frac{d}{d\theta}\left[\left|1+\theta z\right|^{p}\left(1+\theta z\right)^l\right]
	&=&\left(\frac p2+l\right)\left|1+\theta z\right|^{p}\left(1+\theta z\right)^{l-1}z \notag\\
	&& +\frac p2\left|1+\theta z\right|^{p-2}\left(1+\theta z\right)^{l+1}\overline z,\notag
	\end{eqnarray}
	and $\frac12\le \left|1+\theta z\right|\le\frac32$ for any $\theta\in[0,1], \left|z\right|\le \frac{1}{2}$, we have
	\begin{eqnarray}
	\left|\frac{d}{d\theta}\left[\left|1+\theta z\right|^{p}\left(1+\theta z\right)^l\right]\right|\le \left(p+l\right)\left|1+\theta z\right|^{p+l-1}\left|z\right|\lesssim \left|z\right|.\notag
	\end{eqnarray}
	This inequality together with (\ref{12131}) gives (\ref{11191}), which in turn proves (\ref{2.0032}).
	
	We  turn now  to the proof of (\ref{3.0021})-(\ref{12164}). Since
	\begin{equation}
	\partial_t \left(|u|^\alpha u\right)=\frac{\alpha+2}2|u|^\alpha\partial_t u+\frac\alpha2|u|^{\alpha-2}u^2\partial_t\overline u,\notag
	\end{equation}
	we can decompose
	\begin{equation}
	\partial_t\left(|u+v|^\alpha(u+v)-|u|^\alpha u\right)=I_1+I_2,\notag
	\end{equation}
	where
	\begin{equation}
	I_1\left(u,v\right)=\frac{\alpha+2}2|u|^\alpha\partial_t v+\frac\alpha2|u|^{\alpha-2}u^2\overline{\partial_t v}\notag
	\end{equation}
	and
	\begin{eqnarray}\label{2.0040}
	I_2\left(u,v\right)&=&\frac{\alpha+2}2\left(|u+v|^\alpha-|u|^\alpha\right)\left(\partial_t u+\partial_t v\right)
	\notag\\
	&&+\frac\alpha2\left(|u+v|^{\alpha-2}(u+v)^2-|u|^{\alpha-2}u^2\right)\left(\overline{\partial_t u}+\overline{\partial_t v}\right).
	\end{eqnarray}
	Clearly, $|I_1|\leq(\alpha+1)|u|^\alpha|\partial_t v|$. Moreover, applying (\ref{2.0040}), (\ref{121710}) and  (\ref{2.0032}), we obtain the estimate of  $I_2$  in (\ref{12164}).

	Choosing $M$ larger enough, we complete the proof of Lemma \ref{nonlinear estimate}
\end{proof}
\section{CONSTRUCTION AND ESTIMATES OF APPROXIMATE SOLUTIONS}\label{section 3}
In this section,  we construct a sequence  of solutions $u_n$ of (\ref{NLS}),  close to the ansatz $U_J$ in Lemma \ref{L2.2J},  which will eventually converge to the blowing-up solution of Theorem \ref{T1.1}. We will estimate $\varepsilon_n=u_n-U_J$ by the energy method and Strichartz estimate. More precisely,  we estimate
$$
(-t)^{-\sigma}\|\varepsilon_n\|_2+(-t)^{-(1-\delta)\sigma}\|\Delta^2\varepsilon_n\|_2+(-t)^{-(1-\frac\delta2)\sigma}\|\partial_t\varepsilon_n\|_2
$$
for some appropriate parameters $\sigma, \delta$ defined in  (\ref{3.4}) and (\ref{3.1}).

Let   the ansatz $U_{J}$  and $T<0$
be given by Lemma \ref{L2.2J}.   Since  $4 J \leq k-6$   { by }(\ref{k}), we see that
$U_{J}\left(-\frac1n\right) \in H^{4}\left(\mathbb{R}^{N}\right)$  by (\ref{2.31J}).  Therefore, we can deduce from Theorems \ref{T1.00} and \ref{T1.01}  that there exist $s_{n}<-\frac1n$
and a unique maximal solution $u_{n} \in C\left(\left(s_{n},   -\frac1n\right],   H^{4}\left(\mathbb{R}^{N}\right)\right) $ to the following nonlinear fourth-order Schr\"odinger equation
\begin{equation}\label{3.7}
\left\{\begin{array}{l}{i\partial_{t} u_{n}+ \Delta^2 u_{n}+\mu\Delta u_n+\lambda \left|u_{n}\right|^\alpha u_n}=0,  \\ {u_{n}\left(-\frac1n\right)=U_{J}\left(-\frac1n\right)}, \end{array}\right.
\end{equation}
with the blowup alternative that if $s_{n}>$
$-\infty,  $ then
\begin{equation}\label{3.06}
\left\|u_{n}(t)\right\|_{H^{4}} \underset{t \downarrow_{s_{n}}}{\longrightarrow} \infty,
\end{equation}
in the subcritical case $\alpha>0$, $(N-8)\alpha<8$, and
\begin{equation}\label{3.006}
\|u_{n}(t)\|_{L^{\fc{2N-8}{N-8}}((s_n,-\fc1n],L^{\fc{2N(N-4)}{(N-8)^2}})}=\infty
\end{equation}
in the critical case $\alpha=\fc8{N-8},N\geq9$. Let  $\varepsilon_{n} \in C\left(\left(\max \left\{s_{n},   T\right\},   -\frac1n\right],   H^{4}\left(\mathbb{R}^{N}\right)\right) $ be
defined by
\begin{equation}\label{3.07}
u_{n}=U_{J}+\varepsilon_{n}.
\end{equation}
We then have the following estimate.
\begin{proposition}\label{P3.1}
	Assume $ T\le T_0<0$ satisfies
	\begin{eqnarray*}
		C_{11}\left(-T_0\right)^{-1+\min\{\alpha,1\}\left(1-\delta\right)\sigma}+C_{12}\left(-T_0\right)^{\alpha \left(1-\delta\right)\sigma}\le \frac{1}{2},
	\end{eqnarray*}
	where $C_{11},C_{12}$ are the constants in (\ref{12154}) and (\ref{12161}) respectively. There exist $T_0 \leq S<0$ and $n_0>-\frac1S$ such that $s_{n}<S,  $ for all $n\geq n_0$.
	Moreover,
	\begin{gather}
	\left\|\varepsilon_{n}(t)\right\|_{L^{2}} \leq(-t)^{\sigma},
	\|\Delta^2 \varepsilon_{n}(t)\|_{L^{2}} \leq(-t)^{(1-\delta) \sigma},\|\partial_t\ep_n(t)\|_2\leq(-t)^{(1-\frac\delta2)\sigma} \label{7219}
	\end{gather}
	for all $n\geq n_0$ and  $t\in[S, -\frac1{n}]$.
\end{proposition}
\begin{proof}
	Throughout the proof,  we write $\varepsilon$ instead of $\varepsilon_n$ to simplify the notation. It follows from  (\ref{2.017}) and (\ref{3.07}) that
	\begin{equation}\label{equation for epsilon}
	\begin{cases}
	i\partial_t\varepsilon+\Delta^2\varepsilon+\mu\Delta\varepsilon+\lambda(|U_J+\varepsilon|^\alpha(U_J+\varepsilon)-|U_J|^\alpha U_J)+\mathcal{E}_J=0, \\
	\varepsilon(-\frac{1}{n})=0.
	\end{cases}
	\end{equation}
	Let
	\begin{equation}\label{3.13}
	\begin{aligned}
	\tau_{n}=\inf \{ &t \in [\max\{T_0,   s_{n}\},  -\frac1n] ;\|\varepsilon(s)\|_{L^{2}} \leq(-s)^{\sigma},\ \|\Delta^2\varepsilon(s)\|_2\leq(-s)^{(1-\delta)\sigma},\\& \|\partial_t\varepsilon(s)\|_2\leq(-s)^{(1-\frac{\delta}{2})\sigma}\text { for all } t<s \leq -\frac1n\}.
	\end{aligned}
	\end{equation}
	Since $\varepsilon(-\frac1n)=0$,  we see that $s_n\leq\tau_n<-\frac1n$. In what follows, we define  $I=(t,-\fc1n)$ with $\tau_n<t<-\fc1n$.
	
	Firstly, we claim that
	\begin{gather}
	\|\partial_t\varepsilon\|_{L^q(I,L^r)}+\|\partial_t\varepsilon\|_{L^\gamma(I,L^\rho)}\lesssim(-t)^{(1-2\delta)\sigma},\label{SSnorm}
	\end{gather}
	where $\left(\gamma ,\rho\right)$ is the admissible pair defined in (\ref{7181}) and
	\begin{equation}\label{721}
	(q,r)=\left(\fc{16(\alpha+1)}\alpha,\fc{4(\alpha+1)}{\alpha+2}\right)\in\Lambda_b
	\end{equation}
	is another given biharmonic admissible pair. In fact, note that
	\begin{gather}
	\varepsilon(t)=i\int_t^{-\frac1n}e^{i\left(t-s\right)(\Delta^2+\mu\Delta)}\left[\lambda(|U_J+\varepsilon|^\alpha(U_J+\varepsilon)-|U_J|^\alpha U_J)+\mathcal{E}_J\right]\left(s\right)\ ds\label{4.0016}
	\end{gather}
	by the equation (\ref{equation for epsilon}) and that
\begin{eqnarray*}
		\partial_t\varepsilon(t)&=&-ie^{i\left(t+\frac{1}{n}\right)(\Delta^2+\mu\Delta)}\mathcal{E}_J(-\frac1n)
		\\&&+i\int_t^{-\frac1n}e^{it(\Delta^2+\mu\Delta)}[\lambda\partial_s(|U_J+\varepsilon|^\alpha(U_J+\varepsilon)-|U_J|^\alpha U_J)+\partial_s\mathcal{E}_J]\left(s\right)\ ds.
\end{eqnarray*}
	Using Strichartz's estimate (\ref{SZ}) and Lemma \ref{nonlinear estimate}, we deduce that
	\begin{eqnarray}\label{4.016}
	\|\partial_t\varepsilon\|_{L^q(I,L^r)\cap L^\gamma(I,L^\rho)}
	&\lesssim&\|\mathcal{E}_J(-\frac1n)\|_2+\|I_1\left(U_J,\ep\right)\|_{L^1(I,L^2)}\notag\\
	&&+\|I_2\left(U_J,\ep\right)\|_{L^{q'}(I,L^{r'})}+\|\partial_t\mathcal{E}_J\|_{L^1(I,L^2)}.
	\end{eqnarray}
	Next, it follows from  (\ref{2.06}), (\ref{2.9}), (\ref{2.29J})-(\ref{2.021J}), (\ref{3.13}) and Lemma \ref{nonlinear estimate} that
	\begin{gather}
	\|\mathcal{E}_J(-\frac1n)\|_2\lesssim(-\frac1n)^{-\frac{1}{\alpha}+J(1-\frac4k)-\frac4k}\lesssim(-t)^{-\frac{1}{\alpha}+J(1-\frac4k)-\frac4k}\lesssim(-t)^{(1-2\delta)\sigma},\label{4.017}\\
	\|I_1\left(U_J,\ep\right)\|_{L^1(I,L^2)}\lesssim\||U_J|^\alpha|\partial_t\ep|\|_{L^1L^2}\lesssim(-t)^{(1-\frac\delta2)\sigma}\lesssim(-t)^{(1-2\delta)\sigma},\label{4.018}\\
	\|\partial_t\mathcal{E}_J\|_{L^1(I,L^2)}\lesssim(-t)^{-1-\frac{1}{\alpha}+J(1-\frac4k)-\frac4k}\lesssim(-t)^{(1-2\delta)\sigma}, \label{4.019}
	\end{gather}
	where we also used (\ref{J})-(\ref{3.4}) to ensure $-1-\frac{1}{\alpha}+J\left(1-\frac{4}{k}\right)-\frac{4}{k}\ge \left(1-2\delta\right)\sigma$.  It remains to estimate $\left\|I_2\left(U_J,\ep\right)\right\|_{L^{q'}\left(I,L^{r'}\right)}$. Note that
	\begin{eqnarray}\label{121711}
	\left|I_2\left(U_J,\ep\right)\right|&\lesssim& \left(\left|U_J\right|^{\alpha-1}\left|\ep\right|+1_{\alpha>1}\left|\ep\right|^{\alpha}\right)\left|\partial_tU_J\right|\notag\\
	&&+\left(\left|\ep\right|^{\alpha}+1_{\alpha>1}\left|U_J\right|^{\alpha-1}\left|\ep\right|\right)\left|\partial_t\ep\right|
	\end{eqnarray} by (\ref{12164}) and that
	\begin{eqnarray}\label{121712}
	\left|U_J\right|^{\alpha-1}\left|\partial_tU_J\right|\lesssim \left(-t\right)^{-1}\left|U_0\right|^{\alpha}
	\end{eqnarray}
	by (\ref{2.021J}) and (\ref{2.02J}).  It follows from (\ref{121711}), (\ref{121712}), (\ref{2.021J}), (\ref{2.02J}) and H\"older's inequality that
	\begin{eqnarray}\label{121713}
	\left\|I_2\left(U_J,\ep\right)\right\|_{L^{r'}}&\lesssim& \left(-t\right)^{-1}\left\|U_0\right\|_{2\alpha+2}^\alpha \left\|\ep\right\|_r
	+1_{\alpha>1}\left(-t\right)^{-1}\left\|\ep\right\|_{2\alpha+2}^\alpha \left\|U_0\right\|_r\notag\\
	&&+\left\|\ep\right\|_{2\alpha+2}^\alpha \left\|\partial_t\ep\right\|_r+1_{\alpha>1}\left\|U_0\right\|_{2\alpha+2}^{\alpha-1}\left\|\ep\right\|_{2\alpha+2}\left\|\partial_t\ep\right\|_r.
	\end{eqnarray}
	Applying (\ref{121713}),    (\ref{2.9}),  (\ref{3.13}) and  Sobolev's embedding $H^4\hookrightarrow L^{2\alpha+2}\cap L^r$, we conclude that
	\begin{eqnarray}
	\left\|I_2\left(U_J,\ep\right)\right\|_{L^{r'}}&\lesssim& \left(-t\right)^{-2+\left(1-\delta\right)\sigma}+1_{\alpha>1}\left(-t\right)^{-1-\frac{1}{\alpha}+\alpha \left(1-\delta\right)\sigma}\notag\\
	&&+\left(-t\right)^{\alpha \left(1-\delta\right)\sigma}\left\|\partial_t\ep\right\|_r+1_{\alpha>1}\left(-t\right)^{-1+\frac{1}{\alpha}+\left(1-\delta\right)\sigma}\left\|\partial_t\ep\right\|_r\notag\\
	&\lesssim &\left(-t\right)^{-2+\left(1-\delta\right)\sigma}+\left(-t\right)^{-1+\min\{\alpha,1\}\left(1-\delta\right)\sigma}\left\|\partial_t\ep\right\|_r\notag
	\end{eqnarray}
	Next, taking $L^{q'}\left(I\right)$ norm in the above inequality, and then  applying H\"older's inequality,  we obtain
	\begin{eqnarray}\label{4.020}
	&&\left\|I_2\left(U_J,\ep\right)(t)\right\|_{L^{q'}\left(I,L^{r'}\right)}\notag\\
	&\lesssim &\left(-t\right)^{-1-\frac{1}{q}+\left(1-\delta\right)\sigma}+\left(-t\right)^{-\frac{2}{q}+\min\{\alpha,1\}\left(1-\delta\right)\sigma}\left\|\partial_t\ep\right\|_{L^q\left(I,L^r\right)}\notag\\
	&\lesssim &\left(-t\right)^{\left(1-2\delta\right)\sigma}+\left(-t\right)^{-1+\min\{\alpha,1\}\left(1-\delta\right)\sigma}\left\|\partial_t\ep\right\|_{L^q\left(I,L^r\right)},
	\end{eqnarray}
	where we used  (\ref{3.4}) to ensure $-1-\frac{1}{q}+\left(1-\delta\right)\sigma\ge \left(1-2\delta\right)\sigma$.
	Putting  (\ref{4.016})-(\ref{4.020}) together, we deduce that
	\begin{eqnarray}\label{12154}
	&&\|\partial_t\varepsilon\|_{L^q(I,L^r)\cap L^\gamma(I,L^\rho)}\notag\\
	&\le& C_{11}\left(-t\right)^{\left(1-2\delta\right)\sigma}+ C_{11}\left(-t\right)^{-1+\min\{\alpha,1\}\left(1-\delta\right)\sigma}\left\|\partial_t\ep\right\|_{L^q\left(I,L^r\right)}.
	\end{eqnarray}
	Since $C_{11}\left(-t\right)^{-1+\min\{\alpha,1\}\left(1-\delta\right)\sigma}\le \frac{1}{2}$, we obtain (\ref{SSnorm}).

	Next,  we claim that $s_n<\tau_n$. In fact, the subcritical case $\alpha>0$, $(N-8)\alpha<8$ follows immediately from the blowup alternative (\ref{3.06}) and the definition of $\tau_n$ in (\ref{3.13}). For the  critical case $\alpha=\fc8{N-8},N\geq9$, it suffices to prove that
	\begin{eqnarray}\label{12204}
	\|\ep\|_{L^{\fc{2N-8}{N-8}}(I,L^{\fc{2N(N-4)}{(N-8)^2}})}\le C<\infty
	\end{eqnarray}
	by considering the blowup alternative (\ref{3.006}).
	
	We now prove (\ref{12204}). Applying the equation (\ref{equation for epsilon}), Lemma \ref{nonlinear estimate} and the same argument used to prove (\ref{12203}),  we conclude that
	\begin{eqnarray}\label{71868}
	&&\|\Delta^2\ep\|_{L^{\gamma}(I,L^\rho)}\notag\\
	&\lesssim& \|\ep\|_{H^{1,\gamma }(I,L^\rho)}+\|(|U_J|^\alpha+|\ep|^\alpha)\ep\|_{L^\gamma(I,L^\rho)}+\left\|\mathcal{E}_J\right\|_{L^\gamma \left(I,L^\rho\right)}
	\end{eqnarray}
	Moreover, using the same method as that used to derive (\ref{11215}) we obtain
	\begin{eqnarray}\label{12246}
	\left\|\left|\ep\right|^{\alpha}\ep\right\|_{L^\gamma \left(I,L^\rho\right)} &\lesssim& \left\|\ep\right\|_{L^\infty \left(I,L^{2\left(\alpha+1\right)}\right)}^\alpha \left(\left\|\ep\right\|_{L^\gamma L^\rho}+\left\|\Delta^2\ep\right\|_{L^\gamma L^\rho}\right),
	\end{eqnarray}
	where we used the embedding (\ref{a1}). Applying (\ref{12246}), Sobolev's embedding $H^4 \hookrightarrow L^\rho \cap L^{2\left(\alpha+1\right)}$ and (\ref{3.13}), we deduce that
	\begin{eqnarray}\label{71869}
	\left\|\left|\ep\right|^{\alpha}\ep\right\|_{L^\gamma \left(I,L^\rho\right)}
	&\lesssim &\left(-t\right)^{\alpha\left(1-\delta\right)\sigma}\left(\left(-t\right)^{\frac{1}{\gamma }+\left(1-\delta\right)\sigma}+  \left\|\Delta^2\ep\right\|_{L^\gamma \left(I,L^\rho\right)}\right)\notag\\
	&\lesssim& (-t)^{\frac{1}{\gamma }+\left(\alpha+1\right)\left(1-\delta\right)\sigma}+\left(-t\right)^{\alpha \left(1-\delta\right)\sigma}\left\|\Delta^2\ep\right\|_{L^\gamma \left(I,L^\rho\right)},
	\end{eqnarray}
	On the other hand, since $\left|U_J\right|^{\alpha}\lesssim \left(-t\right)^{-1}$ by (\ref{2.06}) and (\ref{2.021J}), we can apply  (\ref{2.9}), (\ref{2.29J}),   Sobolev's embedding $H^4 \hookrightarrow L^\rho$ and (\ref{3.13})  to obtain
	\begin{eqnarray}\label{12156}
	&&\left\|\left|U_J\right|^{\alpha}\ep\right\|_{L^\gamma \left(I,L^\rho\right)}+\left\|\mathcal{E}_J\right\|_{L^\gamma \left(I,L^\rho\right)}\notag\\
	&\lesssim& \left\|\left(-s\right)^{-1+\left(1-\delta\right)\sigma}\right\|_{L^\gamma \left(I\right)}+\left\|\left(-s\right)^{-\frac{1}{\alpha}+J\left(1-\frac{4}{k}\right)-\frac{4}{k}}\right\|_{L^\gamma \left(I\right)}\notag\\
	&\lesssim& \left(-t\right)^{-1+\left(1-\delta\right)\sigma}
	\end{eqnarray}
	where we  used (\ref{J})--(\ref{3.4}) to ensure $-\frac{1}{\alpha}+J\left(1-\frac{4}{k}\right)-\frac{4}{k}\ge-1+\left(1-\delta\right)\sigma$ in the last inequality. It follows from (\ref{SSnorm}), (\ref{71868}),  (\ref{71869}) and (\ref{12156}) that
	\begin{eqnarray}\label{12161}
	\|\Delta^2\ep\|_{L^{\gamma}(I,L^\rho)}\leq C_{12}(-t)^{-1+(1-2\delta)\sigma}+C_{12}\left(-t\right)^{\alpha \left(1-\delta\right)\sigma}\left\|\Delta^2\ep\right\|_{L^\gamma \left(I,L^\rho\right)}.
	\end{eqnarray}
	Since $C_{12}\left(-t\right)^{\alpha \left(1-\delta\right)\sigma}\le \frac{1}{2}$, we deduce that
	\begin{equation}
	\|\Delta^2\ep\|_{L^{\gamma}(I,L^\rho)}\lesssim (-t)^{-1+\left(1-2\delta\right)\sigma}.\notag
	\end{equation}
	This inequality together with Sobolev's embedding $\dot H^{4,\rho}(\R^N)\hookrightarrow L^{\fc{2N(N-4)}{(N-8)^2}}$ $ (\R^N)$ yields (\ref{12204}).
	
	We now resume the proof of Proposition \ref{P3.1}. We first estimate $\|\varepsilon(t)\|_{L^2}$. Multiplying (\ref{equation for epsilon}) by $\overline\varepsilon$ and then taking the imaginary part,  we obtain
	\begin{equation*}
	\begin{aligned}
	\frac12\frac{d}{dt}\|\varepsilon(t)\|_{L^2}^2=-\text{Im}\left(\lambda\int \left[|U_J+\varepsilon|^\alpha(U_J+\varepsilon)-|U_J|^\alpha U_J\right]\overline\varepsilon\right)-\text{Im}\int\mathcal{E}_J\bar\varepsilon.
	\end{aligned}
	\end{equation*}
	Using Lemma \ref{nonlinear estimate},   we deduce that
	\begin{equation}\label{3.15}
	\frac12\frac{d}{dt}\|\varepsilon(t)\|_{L^2}^2\geq-|\lambda|M\int(|U_J|^\alpha
	+|\varepsilon|^\alpha)|\varepsilon|^2-\|\mathcal{E}_J\|_{L^2}\|\varepsilon\|_{L^2}.
	\end{equation}
	By (\ref{2.06}) and (\ref{2.021J})
	\begin{equation}\label{3.16}
	\int|U_J|^\alpha|\varepsilon|^2\leq2^\alpha(-\alpha\text{Im}\lambda)^{-1}(-t)^{-1}\|\varepsilon\|_{L^2}^2.
	\end{equation}
	Moreover, it follows from  Sobolev's embedding $H^4(\R^N)\hookrightarrow L^{\alpha+2}(\R^N)$, (\ref{3.13}) and (\ref{2.29J}) that
	\begin{equation}\label{12152}
	\int|\varepsilon|^{\alpha+2}\lesssim\|\varepsilon\|_{H^4}^{\alpha+2}\lesssim(-t)^{(\alpha+2)(1-\delta)\sigma}\lesssim \left(-t\right)^{2\sigma},
	\end{equation}
	and
	\begin{equation}\label{12153}
	\left\|\mathcal{E}_{J}\right\|_{L^{2}}\|\varepsilon\|_{L^{2}} \lesssim(-t)^{J\left(1-\frac{4}{k}\right)-\frac{4}{k}-\frac{1}{\alpha}+\sigma}\lesssim \left(-t\right)^{2\sigma},
	\end{equation}
	where we used (\ref{J})--(\ref{3.4}) to ensure $\min\{\left(\alpha+2\right)\left(1-\delta\right)\sigma,J\left(1-\frac{4}{k}\right)-\frac{4}{k}-\frac{1}{\alpha}+\sigma\}\ge2\sigma$.

	Applying   (\ref{3.15}),  (\ref{3.16}), (\ref{12152}) and (\ref{12153}), we conclude that
	\begin{equation*}
	\frac{d}{d t}\|\varepsilon(t)\|_{L^{2}}^{2} \geq-2^{\alpha+1}( -\alpha\text{Im}\lambda)^{-1}|\lambda| M(-t)^{-1}\|\varepsilon\|_{L^{2}}^{2}-C(-t)^{2 \sigma},
	\end{equation*}
	and so
	\begin{equation*}
	\begin{aligned} \frac{d}{d t}\left((-t)^{-\sigma}\|\varepsilon(t)\|_{L^{2}}^{2}\right) &=\sigma(-t)^{-\sigma-1}\|\varepsilon(t)\|_{L^{2}}^{2}+(-t)^{-\sigma} \frac{d}{d t}\|\varepsilon(t)\|_{L^{2}}^{2} \\
	& \geq\left[\sigma-2^{\alpha+1} (-\alpha\text{Im}\lambda)^{-1}|\lambda|M\right](-t)^{-\sigma-1}\|\varepsilon(t)\|_{L^{2}}^{2}-C(-t)^{\sigma}.
	\end{aligned}
	\end{equation*}
	Using $\ep \left(-\frac{1}{n}\right)=0$ and $\sigma\ge2^{\alpha+1} (-\alpha\text{Im}\lambda)^{-1}|\lambda|M$ by (\ref{3.4}), we can integrate the above inequality on the interval  $(t,  -\frac1n)$ to obtain
	\begin{equation}\label{3.023}
	\|\varepsilon(t)\|_{L^{2}} \leq C_{13}(-t)^{\frac{1}{2}+\sigma}
	\end{equation}
	for all $t\in(\tau_n,  -\frac1n). $
	
	Our next step is to estimate  $\|\Delta^2\varepsilon\|_2$. Analogously to (\ref{71868}), we can use the equation (\ref{equation for epsilon}) and Lemma \ref{nonlinear estimate} to obtain
	\begin{eqnarray}
	\|\Delta^2\varepsilon\|_2&\lesssim&\|\partial_t\ep\|_2+\|\ep\|_2+\left\|\left(\left|U_J\right|^{\alpha}+\left|\ep\right|^{\alpha}\right)\left|\ep\right|\right\|_2+\left\|\mathcal{E}_J\right\|_{2}.\notag
	\end{eqnarray}
It follows from (\ref{2.06}), (\ref{2.9}), (\ref{2.29J}), (\ref{2.021J}), (\ref{3.13}) and the embedding $H^4\hookrightarrow L^{2\left(\alpha+1\right)}$ that
	\begin{eqnarray}\label{12169}
	\|\Delta^2\varepsilon\|_2
	&\lesssim &\left(-t\right)^{\left(1-\frac{\delta}{2}\right)\sigma}+\left(-t\right)^{-1+\sigma}+(-t)^{\left(\alpha+1\right)\left(1-\delta\right)\sigma}+(-t)^{-\frac{1}{\alpha}+J\left(1-\frac{4}{k}\right)-\frac{4}{k}}\notag\\
	&\le & C_{14}\left(-t\right)^{\left(1-\delta\right)\sigma},
	\end{eqnarray}
	where we used (\ref{J})-(\ref{3.4}) to ensure $\min\{-1+\sigma,-\frac{1}{\alpha}+J\left(1-\frac{4}{k}-\frac{4}{k}\right)\}\ge \left(1-\delta\right)\sigma$.
	
	We now estimate $\|\partial_t\ep\|_{L^2}. $ Applying $\partial_t$ to  the  equation (\ref{equation for epsilon}), and then multiplying it by $\overline{\partial_t\varepsilon}$, integrating by parts, we obtain by applying Lemma \ref{nonlinear estimate}
	\begin{eqnarray}\label{12165}
	&&\frac12\frac{d}{dt}\|\partial_t\varepsilon\|_2^2\notag\\
	&=&-\text{Im}\lambda\int\partial_t(|U_J+\varepsilon|^\alpha (U_J+\varepsilon)-|U_J|^\alpha U_J)\overline{\partial_t\varepsilon}
	-\text{Im}\int \partial_t\mathcal{E}_J\overline{\partial_t\varepsilon}\notag\\
	&=&-\text{Im}\lambda\int I_1\left(U_J,\ep\right)\overline{\partial_t\varepsilon}-\text{Im}\lambda\int I_2\left(U_J,\ep\right)\overline{\partial_t\varepsilon}-\text{Im}\int \partial_t\mathcal{E}_J\overline{\partial_t\varepsilon}.
	\end{eqnarray}
	These formal calculations can be justified by standard approximation arguments, see e.g. Proposition 3.1 in \cite{xuan}. Applying Cauchy-Schwartz inequality, (\ref{2.9}), (\ref{2.01J}) and   (\ref{3.13}),  we obtain
	\begin{eqnarray}\label{12166}
	|\int \partial_t\mathcal{E}_J\overline{\partial_t\varepsilon}|\lesssim(-t)^{-1+J(1-\frac4k)-\frac4k+(1-\frac\delta2)\sigma}\lesssim(-t)^{2\sigma},
	\end{eqnarray}
	where we used (\ref{J}) to ensure $-1+J(1-\frac4k)-\frac4k+(1-\frac\delta2)\sigma\ge2\sigma$.
	On the other hand, it follows from  (\ref{2.06}), (\ref{2.021J}) and (\ref{12163}) that
	\begin{eqnarray}\label{12167}
	|\text{Im}\lambda\int I_1\left(U_J,\ep\right)\overline{\partial_t\varepsilon}|\leq|\lambda|M\int|U_J|^\alpha
	|\partial_t\ep|^2\leq2^\alpha|\lambda|M(-\alpha\text{Im}\lambda)^{-1}\|\partial_t\ep\|_2^2.
	\end{eqnarray}
	It remains to estimate $-\text{Im}\lambda\int I_2\left(U_J,\ep\right)\overline{\partial_t\varepsilon}$. Using  (\ref{121711}), (\ref{121712}), (\ref{2.021J}), (\ref{2.02J}) and H\"older's inequality, we conclude that
	\begin{eqnarray}\label{12168}
	&&|-\text{Im}\lambda\int I_2\left(U_J,\ep\right)\overline{\partial_t\varepsilon}|\notag\\
	&\lesssim&\left(-t\right)^{-1}\left\|U_0\right\|_\infty ^\alpha \left\|\ep\right\|_2 \left\|\partial_t\ep\right\|_2+1_{\alpha>1}\left(-t\right)^{-1}\left\|\ep\right\|_{2\alpha+2}^\alpha \left\|U_0\right\|_r \left\|\partial_t\ep\right\|_r\notag\\
	&&+\left\|\ep\right\|_{2\alpha+2}^\alpha \left\|\partial_t\ep\right\|_r^2+1_{\alpha>1}\left\|U_0\right\|_{2\alpha+2}^{\alpha-1}\left\|\ep\right\|_{2\alpha+2}\left\|\partial_t\ep\right\|_{r}^2
	\end{eqnarray}
	Moreover, it follows from  (\ref{12168}),    (\ref{2.9}),  (\ref{3.13}) and  Sobolev's embedding $H^4\hookrightarrow L^{2\alpha+2}\cap L^r$ that
	\begin{eqnarray}\label{121714}
	&&|-\text{Im}\lambda\int I_2\left(U_J,\ep\right)\overline{\partial_t\varepsilon}|\notag\\
	&\lesssim& \left(-t\right)^{-2+\left(2-\frac{\delta}{2}\right)\sigma}+1_{\alpha>1}\left(-t\right)^{-1-\frac{1}{\alpha}+\alpha \left(1-\delta\right)\sigma}\left\|\partial_t\ep\right\|_r \notag\\
	&&+\left(-t\right)^{\alpha \left(1-\delta\right)\sigma}\left\|\partial_t\ep\right\|_r^2+1_{\alpha>1}\left(-t\right)^{-1+\frac{1}{\alpha}+\left(1-\delta\right)\sigma}\left\|\partial_t\ep\right\|_r^2  \notag\\
	&\lesssim &\left(-t\right)^{-2+\left(2-\frac{\delta}{2}\right)\sigma}+1_{\alpha>1}\left(-t\right)^{-2+\alpha \left(1-\delta\right)\sigma}\left\|\partial_t\ep\right\|_r\notag\\
	&&+\left(-t\right)^{-1+\min\{\alpha,1\}\left(1-\delta\right)\sigma}\left\|\partial_t\ep\right\|_r^2.
	\end{eqnarray}
	It now  follows from (\ref{12165}), (\ref{12166}), (\ref{12167}) and (\ref{121714}) that
	\begin{eqnarray*}
		\frac{d}{dt}\|\partial_t\varepsilon\|_2^2&\geq&-2^{\alpha+1}|\lambda|M(-\alpha\text{Im}\lambda)^{-1}(-t)^{-1}\|\partial_t\varepsilon\|_2^2-C(-t)^{-2+(2-\frac\delta2)\sigma}\\
		&&-C1_{\alpha>1}(-t)^{-2+\alpha(1-\delta)\sigma}\|\partial_t\ep\|_r
		-C\left(-t\right)^{-1+\min\{\alpha,1\}\left(1-\delta\right)\sigma}\left\|\partial_t\ep\right\|_r^2,
	\end{eqnarray*}
	so that
	\begin{eqnarray*}
		&&\frac{d}{dt}\left[(-t)^{-\frac12\min\{\alpha,1\}(1-\delta)\sigma}\|\partial_t\varepsilon\|_2^2\right]\\
		&\geq&\left(\frac12\min\{\alpha,1\}(1-\delta)\sigma-2^{\alpha+1}|\lambda|M(-\alpha\text{Im}\lambda)^{-1}\right)\notag\\
		&&\qquad\times(-t)^{-1-\frac12\min\{\alpha,1\}(1-\delta)\sigma}\|\partial_t\varepsilon\|_2^2-C1_{\alpha>1}\left(-t\right)^{-2+\left(\alpha-\frac{1}{2}\right) \left(1-\delta\right)\sigma}\|\partial_t\varepsilon\|_r\\
		&&-C(-t)^{-2+(2-\frac\delta2)\sigma-\frac12\min\{\alpha,1\}(1-\delta)\sigma}-C(-t)^{-1+\frac12\min\{\alpha,1\}(1-\delta)\sigma}\|\partial_t\ep\|_r^2.
	\end{eqnarray*}
	Using  $\ep \left(-\frac{1}{n}\right)=0$ and $\frac12\min\{\alpha,1\}(1-\delta)\sigma\ge2^{\alpha+1}|\lambda|M(-\alpha\text{Im}\lambda)^{-1}$ by (\ref{3.4}), we can integrate the above inequality on   $(t,  -\frac1n)$, and then apply H\"older's inequality  to obtain
	\begin{eqnarray}
	 \|\partial_t\varepsilon\|_2^2 
	&\lesssim&(-t)^{-1+(2-\frac\delta2)\sigma}+1_{\alpha>1}(-t)^{-1-\frac1q+\alpha(1-\delta)\sigma}\|\partial_t\ep\|_{L^qL^r}
\notag\\&&+(-t)^{-\frac2q+\min\{\alpha,1\}(1-\delta)\sigma}\|\partial_t\ep\|_{L^qL^r}^2\nonumber\\
	&\lesssim &\left(-t\right)^{-1+\left(2-\frac{\delta}{2}\right)\sigma}+1_{\alpha>1}\left(-t\right)^{-2+\alpha \left(1-\delta\right)\sigma+\left(1-2\delta\right)\sigma} \notag\\
	&&+(-t)^{-1+\min\{\alpha,1\}(1-\delta)\sigma+2\left(1-2\delta\right)\sigma}\notag
	\end{eqnarray}
	where in the last inequality we used (\ref{SSnorm}). Note that
	\begin{eqnarray*}
		&&\min \left\{-1+\left(2-\frac{\delta}{2}\right)\sigma,-1+\min\{\alpha,1\}(1-\delta)\sigma+2\left(1-2\delta\right)\sigma\right\}\notag\\
		&\ &\qquad\ge 2\left(1-\frac{\delta}{2}\right)\sigma+\frac{\delta\sigma}{4},
	\end{eqnarray*}
	and
	\begin{eqnarray*}
		-2+\alpha \left(1-\delta\right)\sigma+\left(1-2\delta\right)\sigma\ge2\left(1-\frac{\delta}{2}\right)\sigma+\frac{\delta\sigma}{4},\qquad \alpha>1
	\end{eqnarray*}
	by (\ref{3.1}) and (\ref{3.4}), so that
	\begin{eqnarray}\label{partialtepsilon}
	\left\|\partial_t\ep\right\|_{2}^2\le C_{15} \left(-t\right)^{2\left(1-\frac{\delta}{2}\right)\sigma+\frac{\delta\sigma}{4}}.
	\end{eqnarray}
	Set $S\in[T_0, 0)$ satisfying
	\begin{equation}\label{12211}
	C_{13}(-S)^{\frac12}\leq\frac12,\ C_{14}(-S)^{\frac{\sigma\delta}{2}}\leq\frac12
	,\ C_{15}(-S)^{\frac{\delta\sigma}4}\leq\frac12.
	\end{equation}
	It follows from (\ref{3.023}), (\ref{12169}), (\ref{partialtepsilon}) and (\ref{12211})  that for $n$ sufficiently large, we have  $S<-\frac1n$ and
	\begin{equation}\label{12212}
	\begin{aligned}
	\|\varepsilon\|_{L^2}\leq \frac{1}{2}(-t)^\sigma,\ \|\Delta^2\varepsilon\|_{L^2}\leq \frac{1}{2}(-t)^{(1-\delta)\sigma},\  \|\partial_t\ep\|_2\leq \frac{1}{2}(-t)^{(1-\frac\delta2)\sigma},
	\end{aligned}
	\end{equation}
	for all $\max\{\tau_n,S\}<t<-\frac1n$.  Suppose $S\le\tau_n$, then (\ref{12212}) holds for all $\tau_n<t<-\frac{1}{n}$. This contradicts  the definition (\ref{3.13}) of $\tau_n$ since $s_n<\tau_n$.     Therefore, we deduce that $\tau_n<S$ and (\ref{12212}) holds for all $S<t<-\frac{1}{n}$, which completes the proof of Proposition \ref{P3.1}.
\end{proof}
\section{PROOF OF THEOREM 1.3}\label{section 4}
Firstly, we show that there exists a solution $u\in  C([S, 0), H^4(\mathbb{R}^N))\cap C^1([S, 0),$ $ L^2(\mathbb{R}^N))$ to the equation (\ref{NLS}). Given $\tau\in (S,  0)$,  it follows from (\ref{7219}) that $\{\varepsilon_n\}_{n\geq\frac{1}{|\tau|}}$ is bounded in $ L^\infty([S, \tau], $ $ H^4(\mathbb{R}^N))\cap W^{1, \infty}([S, \tau], L^2(\mathbb{R}^N))$.  Therefore, after possibly extracting a subsequence, there exists $\varepsilon\in L^\infty([S, \tau], $ $ H^4(\mathbb{R}^N))\cap W^{1, \infty}([S, \tau], L^2(\mathbb{R}^N))$
such that
\begin{eqnarray}
{\varepsilon_{n} \underset{n \rightarrow \infty}{\longrightarrow} \varepsilon, \  \text { weak }^{\star}} \text { in } L^{\infty}\left([S,  \tau],   H^{4}\left(\mathbb{R}^{N}\right)\right), \label{4.4} \\
{\partial_{t} \varepsilon_{n} \underset{n \rightarrow \infty}{\longrightarrow} \partial_{t} \varepsilon, \   \text { weak }^{\star}} \text { in } L^{\infty}\left([S,  \tau],  L^2\left(\mathbb{R}^{N}\right)\right).\label{12251}
\end{eqnarray}
On the other hand, note that   $H^4(\Omega)\hookrightarrow\hookrightarrow L^{\alpha+2}(\Omega)\hookrightarrow  L^2(\Omega)$ and  $\{\varepsilon_n\}_{n\geq\frac{1}{|\tau|}}$ is uniformly bounded in $L^\infty([S, \tau], H^4(\Omega))\cap W^{1, \infty}([S, \tau],  L^2(\Omega))$,   we have (after  extracting a subsequence)
\begin{equation}\label{4.6}
\varepsilon_n\underset{n \rightarrow \infty}{\longrightarrow}\varepsilon  \textrm{ in }\  L^\infty([S, \tau], L^{\alpha+2}(\Omega))
\end{equation}
by Aubin-Lions Theorem,  see Simon \cite{Si}.

Since $\tau\in \left(S,0\right)$ is arbitrary, a standard argument of diagonal extraction shows
that there exists $\varepsilon\in L^\infty_{loc}([S, 0),  H^{4}(\mathbb{R}^N))\cap W^{1, \infty}_{loc}([S, 0), L^2($ $\mathbb{R}^N))$,
such that (after  extracting a subsequence) (\ref{4.4}), (\ref{12251}) and (\ref{4.6}) hold for all $S<\tau<0$. Moreover, (\ref{7219}), (\ref{4.4}), (\ref{12251}) and (\ref{4.6}) imply that
\begin{gather}
\left\|\varepsilon(t)\right\|_{L^{2}} \leq(-t)^{\sigma},
\left\|\Delta^2 \varepsilon(t)\right\|_{L^{2}} \leq(-t)^{(1-\delta) \sigma},\left\|\partial_t \varepsilon(t)\right\|_{L^{2}} \leq(-t)^{(1-\frac\delta2) \sigma}, \quad
\label{4.6e}
\end{gather} for  $S\leq t<0$. In addition
it follows easily from (\ref{equation for epsilon}) and the convergence properties (\ref{4.4}), (\ref{12251}) and (\ref{4.6}) that
\begin{equation}\label{4.7}
i\partial_t\varepsilon+\Delta^2\ep +
\mu\Delta\varepsilon+\lambda(|U_J+\varepsilon|^\alpha(U_J+\varepsilon)-|U_J|^\alpha U_J)+\mathcal{E}_J=0
\
\end{equation}
in $\ L^\infty_{loc}([S, 0),  L^{2}(\mathbb{R}^N)).$
Therefore,  setting
\begin{equation}\label{4.8}
u(t)=U_J(t)+\varepsilon(t), \quad S\leq t<0,
\end{equation}
we see that $
u \in L_{\mathrm{loc}}^{\infty}\left([S,  0),  H^{4}\left(\mathbb{R}^{N}\right)\right) \cap H_{\mathrm{loc}}^{1,  \infty}\left([S,  0),  L^2\left(\mathbb{R}^{N}\right)\right)$ and that
\begin{equation*}
i\partial_tu+\Delta^2 u+\mu\Delta u+\lambda|u|^\alpha u=0,
\ \textrm{ in }\ L^\infty_{loc}([S, 0), L^{2}(\mathbb{R}^N))
\end{equation*}
by (\ref{epJ}),  (\ref{4.7}) and (\ref{4.8}). By the local existence in $H^4(\mathbb{R}^N)$ and the uniqueness in $L^\infty_tH^4_x$,  we conclude that $u\in  C([S, 0), H^4(\mathbb{R}^N))\cap C^1([S, 0),$ $ L^2(\mathbb{R}^N))$.

Next, we  prove properties (\ref{1.4})-(\ref{1.6}) in Theorem \ref{T1.1}. We first prove (\ref{1.6}). Let $\Omega$ be an open subset of $\mathbb{R}^N$ such that $\overline\Omega\cap K=\emptyset$. It follows from (\ref{2.02}) that $A>0$ on $\Omega$ and   $A(x)=|x|^k$ when $|x|>2R$; and so there exists a constant $c>0$,  such that $A(x)\geq c(1+|x|)^k$ on $\Omega$. Moreover using (\ref{2.03}) and (\ref{12261}),  we deduce that
$$
|U_0|\lesssim A(x)^{-\frac{1}{\alpha}}\lesssim (1+|x|)^{-\frac k\alpha},  \qquad\left||U_0|^\alpha U_0\right|\lesssim (1+|x|)^{ -k(1+\fc1\alpha)}, \ \textrm{ on }\ \Omega,$$
and
$$
|\Delta U_0|\lesssim (1+|x|)^{-\frac k\alpha-2}, \qquad|\Delta^2 U_0|\lesssim (1+|x|)^{-\frac k\alpha-4}, \ \textrm{ on }\ \Omega.
$$
Since $(1+|x|)^{-\frac k\alpha}\in L^2(\mathbb{R}^N)$ by (\ref{k}), we deduce from (\ref{2.27J}), (\ref{2.29J}), (\ref{2.021J}) and the equation (\ref{epJ}) that
\begin{equation*}
\limsup_{t\uparrow0}\|U_J\|_{H^4(\Omega)}<\infty,\qquad \limsup_{t\uparrow0}\|\partial_tU_J\|_{L^2(\Omega)}<\infty.
\end{equation*}
It now follows from  (\ref{4.8}) and the $L^\infty([S, 0),$ $ H^4(\mathbb{R}^N))\cap C^1([S, 0),$ $ L^2(\mathbb{R}^N))$ boundedness of  $\varepsilon$ in  (\ref{4.6e}) that (\ref{1.6}) holds.

Our next aim is to prove (\ref{1.4}). Let now $x_0\in K$ and $r>0$,  it follows from (\ref{2.9}),  (\ref{2.012}) and (\ref{2.021J}) that
\begin{equation}\label{4.11}
(-t)^{-\frac1\alpha+\frac N{2k}}\lesssim\|U_J(t)\|_{L^2(|x-x_0|<r)}\lesssim(-t)^{-\frac1\alpha}.
\end{equation}
Using (\ref{4.8}) and the $L^2$ boundedness of $\ep$ from (\ref{4.6e}), we deduce that
\begin{eqnarray*}
	\|u(t)\|_{L^2(|x-x_0|<r)}&\geq&\|U_J(t)\|_{L^2(|x-x_0|<r)}-\|\varepsilon(t)\|_{L^2(|x-x_0|<r)}\\
	&\gtrsim&(-t)^{-\frac1\alpha+\frac N{2k}}-(-t)^\sigma,
\end{eqnarray*}
which proves  (\ref{1.4}) as $-\frac{1}{\alpha}+\frac{N}{2k}<\sigma$ by (\ref{3.4}).

Finally, we come to the proof of  (\ref{1.5}). Since $k$ satisfies $(2+\frac{8\alpha}k)(N-8)<2N$ by (\ref{k}),  we can choose a real number $p$ satisfying
\begin{equation}\label{4.12}
p>2+\frac{8\alpha}k \text{ and } p(N-8)<2N.
\end{equation}
Hence, we can apply (\ref{2.9}),  (\ref{2.012}),  (\ref{2.021J}) and Gagliardo-Nirenberg's inequality to obtain
\begin{equation*}
\begin{aligned}
(-t)^{-\frac1\alpha+\frac N{pk}}&\lesssim\|U_J\|_{L^p(U)}\lesssim\|\Delta^2 U_J\|_{L^2(U)}^{\frac N4(\frac12-\frac1p)}\|U_J\|_{L^2(U)}^{1-\frac N4(\frac12-\frac1p)}\\
&\lesssim\|\Delta^2 U_J\|_{L^2(U)}^{\frac N4(\frac12-\frac1p)}(-t)^{-\frac1\alpha(1-\frac N4(\frac12-\frac1p))},
\end{aligned}
\end{equation*}
so that
\begin{equation*}
(-t)^{\frac{8}{k(p-2)}-\frac1{\alpha}}\lesssim\|\Delta^2 U_J\|_{L^2(U)}.
\end{equation*}
This inequality together with (\ref{4.12}) implies that
\begin{equation}\label{9211}
\lim_{t\uparrow0}\|\Delta^2 U_J\|_{{L^2(U)}}=\infty.
\end{equation}
On the other hand, note that $U_j=U_{j-1}+w_j$ by (\ref{2.017}), so that $\partial_tU_J=\partial_tw_0+\cdots+\partial_tw_J$. Then it follows from  Lemma \ref{L2.2J} that
\begin{equation*}
\begin{aligned}
|\partial_tU_J|&\gtrsim \left|\partial_tw_0\right|-\left|\partial_tw_1\right|-\cdots-\left|\partial_tw_J\right|\\
&\gtrsim(-t)^{-1}|U_0|-(-t)^{-1+(1-\fc4k)}|U_0|-\cdots-(-t)^{-1+J(1-\fc4k)}|U_0|\\ &\gtrsim (-t)^{-1}|U_0|
\end{aligned}
\end{equation*}
provided $t<0$ is sufficiently small. This inequality together with (\ref{2.012}) gives
\[
\|\partial_tU_J\|_{L^2(U)}\gtrsim(-t)^{-1-\fc1\alpha+\fc N{2k}},
\]
which implies that
\begin{equation}\label{9212}
\lim_{t\uparrow0}\|\partial_tU_J\|_{{L^2(U)}}=\infty.
\end{equation}
Combining (\ref{9211}), (\ref{9212}) and the boundedness of $\ep(t)$ in (\ref{4.6e}), we obtain  (\ref{1.5}), thereby completing the  proof of Theorem \ref{T1.1}. \hfill$\Box$
\section{Appendix}\label{Appendix}

\subsection{Proof of Theorem \ref{T1.00} in the case $1\leq N\leq8$}
We fix $\ep>0$ sufficiently small such that
\begin{eqnarray*}
	\left(2-\alpha\right)\ep\le2\alpha,\qquad\ep<2\alpha.
\end{eqnarray*}
Then we define
\begin{eqnarray}\label{121816}
r_\ep=2+\ep,\ \ q_\ep=\fc{8(2+\ep)}{N\ep}.
\end{eqnarray}
It is straightforward to verify that $\left(q_\ep,r_\ep\right)\in\Lambda_b$ is a biharmonic admissible pair and  $2\leq\fc{r_\ep}{r_\ep-2}\alpha<\infty$, which implies that $H^4 \hookrightarrow L^{\frac{r_\ep}{r_\ep-2}\alpha}$.

Choosing $M=2C_{16}(\|\phi\|_{H^4}+\|\phi\|_{H^4}^{\alpha+1})$, and $T>0$ sufficiently small such that
\begin{equation}
\left(C_{16}+C_{17}\right)T^{1-\frac{2}{q_\ep}}M^\alpha\le \frac{1}{2}
\end{equation}
where the constants $C_{16}, C_{17}$ are defined in (\ref{12137}) and (\ref{12138}), respectively.
Set $J=[-T,T]$, and consider the metric space
\begin{equation}\label{12133}
\begin{aligned}
Z_{T,M}=\{&u\in L^\infty(J,H^4)\cap H^{1,q_\ep}(J,L^{r_\ep})\cap H^{1,\infty}(J,L^2):\\
&\|u\|_{L^\infty(J,H^4)\cap L^{q_\ep}(J,L^{r_\ep})}+\|\partial_tu\|_{L^{\infty}(J,L^{2})\cap L^{q_\ep}(J,L^{r_\ep})}\leq M\}
\end{aligned}
\end{equation}
It follows that $Z_{T,M}$ is a complete metric space when equipped with the distance
\begin{equation}\label{12134}
d(u,v)=\|u-v\|_{L^\infty(J,L^2)}+\|u-v\|_{L^{q_\ep}(J,L^{r_\ep})}.
\end{equation}
In what follows,  we  show that the mapping $S$, defined in (\ref{12113}),  is a strict contraction mapping on the space $Z_{T,M}$.

We first show that $S$ maps $Z_{T,M}$ into itself. Applying (\ref{12113}), (\ref{NLS2}),  (\ref{12111}), Strichartz's estimate (\ref{SZ}), H\"older's inequality and  Sobolev's embedding theorem $H^4(\R^N)\hookrightarrow L^{\fc{r_\ep\alpha}{r_\ep-2}}(\R^N)$, we conclude  that
\begin{eqnarray}\label{1812}
&&\|Su\|_{L^\infty(J,L^2)\cap L^{q_\ep}(J,L^{r_\ep})}+\|\partial_t(Su)\|_{L^\infty(J,L^2)\cap L^{q_\ep}(J,L^{r_\ep})}\notag\\
&\lesssim &F\left(\phi,T\right)+\left\|\left|u\right|^{\alpha}\left(\left|u\right|+\left|\partial_tu\right|\right)\right\|_{L^{q_\ep'}\left(J,L^{r_\ep'}\right)} \\
&\lesssim& \|\phi\|_{H^4}+\|\phi\|_{H^4}^{\alpha+1}+T^{1-\fc2{q_\ep}}\|u\|_{L^\infty(J,H^4)}^\alpha \notag 
 \left(\|u\|_{L^{q_\ep}(J,L^{r_\ep})}+\|\partial_tu\|_{L^{q_\ep}(J,L^{r_\ep})}\right).
\end{eqnarray}
Our next step is to  estimate  $\|Su\|_{L^\infty(J,H^4)}$. Using the same method as that used to derive (\ref{1222}), we obtain
\begin{eqnarray}\label{1813}
\left\|\Delta^2 \left(Su\right)\right\|_{L^\infty \left(J,L^2\right)}\lesssim \left\|\partial_t \left(Su\right)\right\|_{L^\infty \left(J,L^2\right)}+\left\|Su\right\|_{L^\infty \left(J,L^2\right)}+\left\|\left|u\right|^{\alpha}u\right\|_{L^\infty \left(J,L^2\right)}.
\end{eqnarray}
To estimate $\left\|Su\right\|_{L^\infty \left(J,L^2\right)}$, we apply  the Fundamental Theorem of Calculus
\begin{eqnarray}\label{12135}
\left|u\right|^{\alpha}u=\left|\phi\right|^{\alpha}\phi+\int_{0}^{t}\partial_s[\left|u\right|^{\alpha}u(s)]\mathrm{d}s,
\end{eqnarray}
H\"older's inequality and Sobolev's embedding $H^4\hookrightarrow L^{2\left(\alpha+1\right)}\cap L^{\frac{2r_\ep\alpha}{r_\ep-2}}$ to obtain
\begin{eqnarray}\label{12136}
\||u|^\alpha u\|_{L^\infty(J,L^2)}&\lesssim& \|\left|\phi\right|^{\alpha}\phi\|_{L^2}+\left\|\partial_t[\left|u\right|^{\alpha}u]\right\|_{L^1\left(J,L^2\right)}\notag\\
&\lesssim & \|\phi\|_{L^{2\left(\alpha+1\right)}}^{\alpha+1}+\|u\|_{L^\infty(J,L^{\frac{2r_\ep\alpha}{r_\ep-2}})}^{\alpha} \|\partial_tu\|_{L^{1}(J,L^{r_\ep})}\notag\\
&\lesssim & \|\phi\|_{H^4}^{\alpha+1}+ T^{1-\frac{1}{q_\ep}}\left\|u\right\|_{L^\infty \left(J,H^4\right)}^\alpha \left\|\partial_tu\right\|_{L^{q_\ep}\left(J,L^{r_\ep}\right)}.
\end{eqnarray}
Applying (\ref{1812}), (\ref{1813}) and (\ref{12136}), we see that if $u\in X_{T,M}$, then
\begin{eqnarray}\label{12137}
&&\|Su\|_{L^\infty(J,H^4)\cap L^{q_\ep}(J,L^{r_\ep})}+\|\partial_t\left(Su\right)\|_{L^{\infty}(J,L^{2})\cap L^{q_\ep}(J,L^{r_\ep})}\notag\\
&\le& C_{16} \left(\|\phi\|_{H^4}+\|\phi\|_{H^4}^{\alpha+1}\right)+C_{16}T^{1-\fc2{q_\ep}}M^{\alpha+1}\le M.
\end{eqnarray}

Our next aim is the desired Lipschitz property of $S$ with respect to the metric $d$ defined in (\ref{12134}). It follows from  Strichartz's estimate (\ref{SZ}), H\"older's inequality and  Sobolev's embedding  $H^4(\R^N)\hookrightarrow L^{\fc{r_\ep\alpha}{r_\ep-2}}(\R^N)$  that, given  $u,v\in Z_{T,M}$,
\begin{eqnarray}\label{12138}
d(Su,Sv) &\lesssim &\||u|^\alpha u-|v|^\alpha v\|_{L^{q_\ep'}(J,L^{r_\ep'})}\notag\\
&\lesssim &T^{1-\frac{2}{q_\ep}}\left(\|u\|_{L^\infty H^4}^\alpha+\|v\|_{L^\infty H^4}^\alpha\right)\|u-v\|_{L^{q_\ep}(J,L^{r_\ep})}\notag\\
&\le& C_{17}T^{1-\frac{2}{q_\ep}}M^\alpha d(u,v)\le \frac{1}{2}d(u,v).
\end{eqnarray}

Therefore, we obtain  a unique solution $u$ in $Z_{T,M}$ by Banach's fixed point theorem. Moreover, using the same method used in Section \ref{s4} and the uniqueness in Appendix, we can extend $u$ to a maximal solution $u\in C\left([0,T_{\text{max}}),H^4\right)$ to the equation (\ref{NLS}) with
\begin{eqnarray*}
	u,u_{t}, \Delta^2 u \in C\left([0, T_{\text{max}}), L^{2}\left(\mathbb{R}^{N}\right)\right) \cap L^{q}_{\text{loc}}\left((0, T_{\text{max}}), L^{r}\left(\mathbb{R}^{N}\right)\right)
\end{eqnarray*} for every biharmonic admissible pair $(q,r)\in\Lambda_b$.  Finally, the bolwup alternative (\ref{bl1}) and the continuous dependence are proved as in Section \ref{s4}.  This completes the proof of Theorem \ref{T1.00} in the case $1\leq N\leq8.$
{\hfill $\square$\medskip}
\subsection{Unconditional uniqueness}
In this subsection, we prove the unconditional uniqueness in $H^4(\R^N)$ for the Cauchy problem (\ref{NLS})-(\ref{12241}). The proof is an obvious adaptation of Propositions 4.2.3 and  4.2.5 in \cite{Ca9}.
\begin{theorem}[Critical case]\label{1191}
	Suppose $\lambda\in\mathbb{C}, \mu=\pm1$ or $0, N\ge9, \alpha=\frac{8}{N-8}, T>0$, and $u^{1}, u^{2} \in C([0,T],H^4(\R^N))$
	are two solutions of (\ref{NLS}) with the same initial data $\phi\in H^4$. It follows that $u^{1}=u^{2}$.
\end{theorem}
\begin{proof}	  Set
	$$
	S=\sup \left\{\tau \in[0, T] ; u^{1}(t)=u^{2}(t) \text { for } 0 \leq t \leq \tau\right\},
	$$
	so that $0 \leq S \leq T$. Uniqueness follows if we show that $S=T$. Assume by contradiction that $S<T$. Changing $u^{1}(\cdot), u^{2}(\cdot)$ to $u^{1}(S+\cdot), u^{2}(S+\cdot),$ we are reduced to the case $S=0$, so that 
	\begin{eqnarray}\label{1152}
	\inf_{\left(q,r\right)\in \Lambda_b}\left\|u^{1}-u^{2}\right\|_{L^{q}\left((0, \tau), L^{r}\right)}>0 \quad \text { for all } \quad 0<\tau \leq T.
	\end{eqnarray}
	On the other hand, it follows from the equation (\ref{12213}) (for both $u^{1}$ and $\left.u^{2}\right)$ that
	$$
	u^{1}(t)-u^{2}(t)=i\lambda\int_{0}^{t} e^{i(t-s) (\Delta^2+\mu\Delta)}\left[\left|u^{1}(s)\right|^{\alpha} u^{1}(s)-\left|u^{2}(s)\right|^{\alpha} u^{2}(s)\right] d s.
	$$
	Applying Strichartz's estimate (\ref{SZ}), we deduce that 
	\begin{eqnarray}\label{12305}
	\sup_{\left(q,r\right)\in \Lambda_b}\left\|u^{1}-u^{2}\right\|_{L^{q}\left((0, \tau), L^{r}\right)}\lesssim \left\|g\left(u^1-u^2\right)\right\|_{L^2\left(\left(0,\tau\right),L^{\frac{2N}{N+4}}\right)}
	\end{eqnarray}
	for every $0<\tau\le T$, where $g=\left|u^1\right|^{\alpha}+\left|u^2\right|^{\alpha}$.  Since  $u^1,u^2\in C\left([0,T],H^4\right)$, we deduce from Sobolev's embedding $H^4 \hookrightarrow L^{\frac{2N}{N-8}}$ that  $u^1,u^2\in C\left([0,T],L^{\frac{2N}{N-8}}\right)$, which implies 
	\begin{eqnarray}\label{12302}
	g\in C\left([0,T],L^{\frac{N}{4}}\right).
	\end{eqnarray}
	Given any $R>0$, we set 
	\begin{eqnarray*}
		g^R=\begin{cases}
			\left|g\right|,\qquad\text{if }\left|g\right|\ge R\\
			0,\quad\qquad\text{if }\left|g\right|< R\
		\end{cases},\qquad g_R=|g|-g^R.
	\end{eqnarray*}
It is not difficult to show (by dominated convergence theorem, using (\ref{12302})) that 
\begin{eqnarray*}
	\left\|g^R\right\|_{L^\infty \left(\left(0,T\right),L^{\frac{N}{4}}\right)}:=\ep_R\underset{R\rightarrow \infty }\longrightarrow0.
\end{eqnarray*}
Moreover, since $\left|g_R\right|\le R$,
\begin{eqnarray*}
	\left\|g_R\right\|_{L^\infty \left(\left(0,T\right),L^{\frac{N}{2}}\right)}\le R^{\frac{1}{2}}\left\|g\right\|_{L^\infty \left(\left(0,T\right),L^{\frac{N}{4}}\right)}^{\frac{1}{2}}=:C(R)<\infty ,
\end{eqnarray*}
for all $R>0$. Therefore, given any $0<\tau\le T$, we deduce from H\"older's inequality that 
\begin{eqnarray}\label{12303}
\left\|g^R\left(u^1-u^2\right)\right\|_{L^2\left(\left(0,\tau\right),L^{\frac{2N}{N+4}}\right)}\le \ep_R \left\|u^1-u^2\right\|_{L^2\left(\left(0,\tau\right),L^{\frac{2N}{N-4}}\right)}
\end{eqnarray}
and 
\begin{eqnarray}\label{12304}
\left\|g_R\left(u^1-u^2\right)\right\|_{L^2\left(\left(0,\tau\right),L^{\frac{2N}{N+4}}\right)}\lesssim C(R)\tau^{\frac{1}{2}}\left\|u^1-u^2\right\|_{L^\infty \left(\left(0,\tau\right),L^2\right)}.
\end{eqnarray}
It follows from (\ref{12305}), (\ref{12303}) and (\ref{12304}) that 
\begin{eqnarray}\label{12306}
&&\sup_{\left(q,r\right)\in\Lambda_b}\left\|u^1-u^2\right\|_{L^q\left(\left(0,\tau\right),L^r\right)}\notag\\
&\le& C_{18} \left[\ep_R+C(R)\tau^{\frac{1}{2}}\right]\sup_{\left(q,r\right)\in\Lambda_b}\left\|u^1-u^2\right\|_{L^q\left(\left(0,\tau\right),L^r\right)}.
\end{eqnarray}
We first fix $R$ sufficiently large so that $C_{18}\ep_R\le \frac{1}{4}$. Then we choose  $0<\tau_0\le T$ sufficiently small so  that $C_{18}C(R)\tau_0^{\frac{1}{2}}\le \frac{1}{4}$, and we deduce from (\ref{12306}) that 
\begin{eqnarray*}
		\sup_{\left(q,r\right)\in\Lambda_b}\left\|u^1-u^2\right\|_{L^q\left(\left(0,\tau\right),L^r\right)}=0.
	\end{eqnarray*} 
This  contradicts (\ref{1152}) and proves the  uniqueness. \end{proof}
\begin{theorem}[Subritical case]\label{12308}
	Suppose $\lambda\in\mathbb{C},N\ge1,\mu=\pm1$ or $0, 0<\alpha,(N-8)\alpha<8, T>0$, and $u^{1}, u^{2} \in L^\infty ([0,T],H^4(\R^N))$
	are two solutions of (\ref{NLS}) with the same initial data $\phi\in H^4$. It follows that $u^{1}=u^{2}$.
\end{theorem}
\begin{proof}
	Similarly to the critical case, it sufficies to prove that there exists $0<\tau\le T$ such that 
	$u^1(t)=u^2(t)$ for any $0\le t\le\tau$.  Let 
	\begin{eqnarray*}
		q=\frac{8\left(\alpha+2\right)}{N\alpha},\qquad r=\alpha+2.
	\end{eqnarray*}
Then it is easy to check that $q>2$ and $\left(q,r\right)\in\Lambda_b$. Moreover, since  $u^1,u^2\in L^\infty \left([0,T],H^4\right)$, we  deduce from   the embedding $H^4 \hookrightarrow L^r$ that $u^1,u^2\in L^q\left([0,T],L^r\right)$. Next, applying Strichartz's estimate and H\"older's inequality, we conclude that 
\begin{eqnarray}
&&\left\|u^1-u^2\right\|_{L^q\left(\left(0,\tau\right),L^r\right)}\notag\\
&\lesssim& \left\|\left|u^1\right|^{\alpha}u^1-\left|u^2\right|^{\alpha}u^2\right\|_{L^{q'}\left(\left(0,\tau\right),L^{r'}\right)}\notag\\
&\lesssim& \tau^{1-\frac{2}{q}}\left(\left\|u^1\right\|_{L^\infty \left(\left(0,\tau\right),L^r\right)}^\alpha+\left\|u^2\right\|_{L^\infty \left(\left(0,\tau\right),L^r\right)}^\alpha\right)\left\|u^1-u^2\right\|_{L^q\left(\left(0,\tau\right),L^r\right)}\notag
\end{eqnarray}
for any $0<\tau\le T$. Moreover, it follows from Sobolev's embedding $H^4 \hookrightarrow L^r$ that 
\begin{eqnarray}\label{12307}
\left\|u^1-u^2\right\|_{L^q\left(\left(0,\tau\right),L^r\right)}&\le&C_{19}\tau^{1-\frac{2}{q}}\left(\left\|u^1\right\|_{L^\infty \left(\left(0,\tau\right),H^4\right)}^\alpha+\left\|u^2\right\|_{L^\infty \left(\left(0,\tau\right),H^4\right)}^\alpha\right)\notag\\
&&\qquad\times\left\|u^1-u^2\right\|_{L^q\left(\left(0,\tau\right),L^r\right)}
\end{eqnarray} Then, we choose $0<\tau_0\le T$ sufficiently small so that 
\begin{eqnarray*}
	C_{19}\tau_0^{1-\frac{2}{q}}\left(\left\|u^1\right\|_{L^\infty \left(\left(0,T\right),H^4\right)}^\alpha+\left\|u^2\right\|_{L^\infty \left(\left(0,T\right),H^4\right)}^\alpha\right)\le \frac{1}{2},
\end{eqnarray*}
and we deduce from  (\ref{12307}) that 	$u^1(t)=u^2(t)$ for any $0\le t\le\tau_0$. This finishes the proof of Theorem \ref{12308}. 
	\end{proof}
\section*{Acknowledgments}
This work is partially supported by the National Natural Science Foundation of China 11771389,  11931010 and 11621101.


\medskip
Received xxxx 2020; revised xxxx 20xx.
\medskip

\end{document}